\documentclass[12pt,twoside,leqno]{article}
\usepackage[T1]{fontenc}
\usepackage{amsfonts}
\usepackage{amsmath,amsthm}
\usepackage{amssymb,latexsym}
\usepackage{enumerate}
\usepackage{tikz}
\usetikzlibrary{arrows,shapes,decorations,decorations.text}

\theoremstyle{plain}
\newtheorem{theorem}{Theorem}[section]
\newtheorem{lemma}[theorem]{Lemma}
\newtheorem{proposition}[theorem]{Proposition}
\newtheorem{corollary}[theorem]{Corollary}

\newtheorem{fact}[theorem]{Fact}
\theoremstyle{definition}

\newtheorem{definition}[theorem]{Definition}
\newtheorem{claim}{Claim}

\theoremstyle{remark}
\newtheorem{remark}[theorem]{Remark}
\newtheorem{question}[theorem]{Question}

\newtheorem{notation}[theorem]{Notation}

\DeclareMathOperator{\fix}{{\rm fix}}
\DeclareMathOperator{\Sym}{{\rm Sym}}
\DeclareMathOperator{\Orb}{{\rm Orb}}
\DeclareMathOperator{\dom}{{\rm dom}}
\DeclareMathOperator{\ran}{{\rm ran}}

\DeclareMathOperator{\TC}{{\rm TC}}

\DeclareMathOperator{\sym}{{\rm sym}}

\DeclareMathOperator{\cl}{{\rm cl}}

\DeclareMathOperator{\Aut}{{\rm Aut}}

\begin{document}
\title{Countable products and countable sums of compact metrizable spaces in the absence of the Axiom of Choice}
\author{Kyriakos Keremedis, Eleftherios Tachtsis and Eliza Wajch\\
Department of Mathematics, University of the Aegean\\
Karlovassi 83200, Samos, Greece\\
kker@aegean.gr\\
Department of Statistics and Actuarial-Financial Mathematics,\\
 University of the Aegean, Karlovassi 83200, Samos, Greece\\
 ltah@aegean.gr\\
Institute of Mathematics\\
Faculty of Exact and Natural Sciences \\
Siedlce University of Natural Sciences and Humanities\\
ul. 3 Maja 54, 08-110 Siedlce, Poland\\
eliza.wajch@wp.pl}
\maketitle
\begin{abstract}
 \medskip
The main aim of the article is to show, in the absence of the Axiom of Choice, relationships between the following, independent of $\mathbf{ZF}$, statements: ``Every countable product of compact metrizable spaces is separable (respectively, compact)'' and ``Every countable product of compact metrizable spaces is metrizable''. Statements related to the above-mentioned ones are also studied. Permutation models (among them new ones) are shown in which a countable sum (also a countable product) of metrizable spaces need not be metrizable, countable unions of countable sets are countable and there is a countable family of non-empty sets of size at most $2^{\aleph_0}$ which does not have a choice function.  A new permutation model is constructed in which every uncountable compact metrizable space is of size at least $2^{\aleph_0}$ but a denumerable family of denumerable sets need not have a multiple choice function.

\end{abstract}

\section{Introduction}
\label{s1}
Before we pass to the main content of the article, let us establish the set-theoretic framework,  notation and basic definitions in Sections \ref{s1.1}--\ref{s1.3}. A brief description of the content of the article is given in Section \ref{s1.4}. All new results of the paper are included in Sections \ref{s2} and \ref{s3}.

\subsection{The set-theoretic framework}
\label{s1.1}

In this paper, the intended context for reasoning and statements of theorems
is the Zermelo-Fraenkel set theory $\mathbf{ZF}$ without the axiom of choice 
$\mathbf{AC}$. The system $\mathbf{ZF+AC}$ is denoted by $\mathbf{ZFC}$. We recommend  \cite{ku1} and \cite{Ku} as a
good introduction to $\mathbf{ZF}$. To stress the fact that a result is proved in $\mathbf{ZF}$ or $\mathbf{ZF+A}$ (where $\mathbf{A}$ is a statement independent of $\mathbf{ZF}$), we shall write
at the beginning of the statements of the theorems and propositions ($%
\mathbf{ZF}$) or ($\mathbf{ZF+A}$), respectively. Apart from models of $\mathbf{ZF}$, we refer to some models
of $\mathbf{ZFA}$ (or $\text{ZF}^0$ in \cite{hr}), that is, we refer also to $\mathbf{ZF}$ with an infinite set of atoms (see \cite{j}, \cite{j1} and \cite{hr}). 

As in \cite{ktw1}, let us recall several facts concerning well-ordered cardinals, permutation models and transferable statements.

We recall that a set $X$ is called \emph{Dedekind-finite} if $X$ there does not exist its proper subset equipotent to $X$. A set that is not Dedekind-finite is called \emph{Dedekind-infinite}.  A \emph{finite ordinal} can be defined as an ordinal number (of von Neumann) which is Dedekind-finite.  A \emph{well-ordered cardinal number} is an initial ordinal number, i.e., an ordinal which is not equipotent to any of its elements. Every well-orderable set is equipotent to a unique well-ordered cardinal number, called the cardinality of the well-orderable set. 
 
As usual, the set of all finite ordinals is denoted by $\omega$. If $n\in\omega$, then $n+1=n\cup\{n\}$. For convenience, we put $\mathbb{N}=\omega\setminus\{0\}$ and call every member of $\mathbb{N}$ a \emph{natural number}. The power set of a set $X$ is denoted by $\mathcal{P}(X)$. A set $X$ is called \emph{countable} if $X$ is equipotent to a subset of $\omega$. A set $X$ is called \emph{uncountable} if $X$ is not countable.  A set $X$ is \emph{finite} if $X$ is equipotent to an element of $\omega$. An \emph{infinite} set is a set which is not finite. An infinite countable set is called \emph{denumerable}. It is customary to denote by $\aleph_0$ the cardinality of every denumerable set. 

A set expressible as a countable union of finite sets is called a \emph{cuf set}.  If $X$ is a set and $\kappa$ is a non-zero well-ordered cardinal number, then $[X]^{\kappa}$ is the family of all subsets of $X$ equipotent to $\kappa$, $[X]^{\leq\kappa}$ is the collection of all subsets of $X$ equipotent to subsets of $\kappa$, and $[X]^{<\kappa}$ is the family of all subsets of $X$ equipotent to a (well-ordered) cardinal number in $\kappa$. 

For sets $X$ and $Y$, 
\begin{itemize}
\item $|X|\leq |Y|$ means that $X$ is equipotent to a subset of $Y$;
\item $|X|=|Y|$ means that $X$ is equipotent to $Y$; 
\item $|X|<|Y|$ means that $|X|\leq |Y|$ and $|X|\neq|Y|$.
\end{itemize} 

For a set $X$, $|X|\leq\aleph_0$ if and only if $X$ is countable, and $|X|\nleq\aleph_0$ if and only if $X$ is uncountable. Furthermore, although we do not use any notion of a cardinal of a not well-orderable set, we can still use, for every set $X$, the following equivalences:
$|X|=|\mathbb{R}|\leftrightarrow |X|=2^{\aleph_0}$, $|X|\leq |\mathbb{R}|\leftrightarrow |X|\leq 2^{\aleph_0}$ and $|X|<|\mathbb{R}|\leftrightarrow |X|<2^{\aleph_0}$.

Since, in Sections \ref{s2} and \ref{s3}, we apply known permutation models of $\mathbf{ZFA}$ and construct a new one (in Section \ref{s3.3}), let us establish our notation concerning constructions and descriptions of such models. We refer to \cite[Chapter 4]{j} and \cite[Chapter 15, p. 251]{j1} for the basic terminology and facts about permutation models.

Suppose we are given a model $\mathcal{M}$ of $\mathbf{ZFA+AC}$ with an infinite set $A$ of all atoms of $\mathcal{M}$, and a group $\mathcal{G}$ of permutations of $A$. For a set $x\in\mathcal{M}$, we denote by $\TC(x)$ the transitive closure of $x$ in $\mathcal{M}$. Every permutation $\phi$ of $A$ extends uniquely to an $\in$-automorphism (usually denoted also by $\phi$) of $\mathcal{M}$. For $x\in \mathcal{M}$, we put:
$$\fix_{\mathcal{G}}(x)=\{\phi\in\mathcal{G}: (\forall t\in x)\phi(t)=t\}\text{ and } \sym_{\mathcal{G}}(x)=\{\phi\in\mathcal{G}: \phi(x)=x\}.$$
Definitions of the concepts of a \emph{normal filter} and a \emph{normal ideal} used below can be found in \cite[Chapter 4, pp. 46--47]{j}. Every (normal) filter of subgroups of the group $\mathcal{G}$ can be called shortly a \emph{(normal) filter on $\mathcal{G}$}. Let us recall the following definition formulated in \cite{ktw1}:

\begin{definition}
\label{s2d9} (Cf. \cite[Definition 2.9]{ktw1}.)
\begin{enumerate}
\item[(i)] The \emph{permutation model} $\mathcal{N}$  \emph{determined by} $\mathcal{M}, \mathcal{G}$ \emph{and a normal filter}  $\mathcal{F}$ of subgroups of $\mathcal{G}$ is defined by the equality:
$$\mathcal{N}=\{x\in\mathcal{M}: (\forall t\in\TC(\{x\}))(\sym_{\mathcal{G}}(t)\in\mathcal{F})\}.$$

\item[(ii)] The \emph{permutation model} $\mathcal{N}$ \emph{determined by} $\mathcal{M}, \mathcal{G}$ \emph{and a normal ideal} $\mathcal{I}$ of subsets of the set of all atoms of $\mathcal{M}$ is defined by the equality:
$$\mathcal{N}=\{x\in\mathcal{M}: (\forall t\in\TC(\{x\}))(\exists E\in\mathcal{I}) (\fix_{\mathcal{G}}(E)\subseteq\sym_{\mathcal{G}}(t))\}.$$
\item[(iii)] (Cf. \cite[p. 46]{j} and \cite[p. 251]{j1}.)  A \emph{permutation model} (or, equivalently, a \emph{Fraenkel--Mostowski model}) is a class $\mathcal{N}$ which can be defined by (i).
\end{enumerate}
\end{definition}

Given a normal ideal $\mathcal{I}$ of subsets of the set $A$ of atoms of $\mathcal{M}$, the filter $\mathcal{F}_{\mathcal{I}}$ on $\mathcal{G}$, generated by $\{\fix_{\mathcal{G}}(E): E\in\mathcal{I}\}$, is a normal filter on $\mathcal{G}$ such that the permutation model determined by $\mathcal{M}, \mathcal{G}$ and $\mathcal{F}_{\mathcal{I}}$ coincides with the permutation model determined by $\mathcal{M}, \mathcal{G}$ and $\mathcal{I}$ (see \cite[p. 47]{j}). For $x\in\mathcal{M}$, a set $E\in\mathcal{I}$ such that $\fix_{\mathcal{G}}(E)\subseteq \sym_{\mathcal{G}}(x)$ is called a \emph{support} of $x$. 

\begin{fact}
\label{f:2} If $\mathcal{V}$ is a Fraenkel--Mostowski model determined by $\mathcal{M}$, a group $\mathcal{G}$ of permutations of the set $A$ of atoms of $\mathcal{M}$ and a normal filter $\mathcal{F}$ on $\mathcal{G}$, then an element $x$ of $\mathcal{V}$ is well-orderable in $\mathcal{V}$ if and only if $\fix_{\mathcal{G}}(x)\in\mathcal{F}$ (see \cite[Equation (4.2), p.47]{j}).
\end{fact}

For the definitions of the terms ``\emph{boundable statement}'' and ``\emph{injectively boundable statement}'' that will be used in the sequel, the reader is referred to \cite{dpi} or \cite[Note 103]{hr}.
\begin{fact}
\label{f:1} (Cf. \cite[p. 722]{dpi} or \cite[Note 103, p. 285]{hr}.) Boundable statements are (up to equivalence)
injectively boundable.
\end{fact}

We recall that a set-theoretic statement $\mathbf{\Phi}$ is \emph{transferable} if there is a metatheorem: \emph{If $\mathbf{\Phi}$ is true in a Fraenkel--Mostowski model of $\mathbf{ZFA}$, then $\mathbf{\Phi}$ is relatively consistent with $\mathbf{ZF}$}. Jech and Sochor showed that boundable statements are transferable (see \cite[Theorems 6.1, 6.8]{j} or \cite{js}), and Pincus showed in \cite{dpi} the stronger result that injectively boundable statements are transferable (see \cite[Metatheorem 3A2]{dpi}, \cite[Note 103, p. 286]{hr} and Theorem \ref{thm:Pinc} below). 

\begin{theorem}
\label{thm:Pinc} (Cf. \cite[Theorem 3A3]{dpi}.) (The Pincus Theorem.) Let $\mathcal{V}_0$ be a Fraenkel-Mostowski model of $\mathbf{ZFA}$. Let $\mathbf{\Phi} $ be a conjunction of injectively
boundable statements each of which is true in $\mathcal{V}_0$. Then there exists a $\mathbf{ZF}$-model $\mathcal{V}$ such that $\mathcal{V}_0\subset \mathcal{V}$, the models $\mathcal{V}$ and $\mathcal{V}_0$ have the same ordinals and their cofinalities, and $\mathbf{\Phi}$ is true in $\mathcal{V}$. Hence, every injectively boundable statement is transferable. 
\end{theorem}

\subsection{Notation and basic definitions related to topology}
\label{s1.2}

Let $\langle X, d\rangle$ be a metric space. Then the $d$-\textit{ball with centre $x\in X$ and radius} $r\in(0, +\infty)$ is the set 
$$B_{d}(x, r)=\{ y\in X: d(x, y)<r\}.$$
 The collection 
$$\tau(d)=\{ V\subseteq X: (\forall x\in V)(\exists r\in (0, +\infty)) B_{d}(x, r)\subseteq V\}$$
is the \textit{topology on $X$ induced by $d$}. We say that the metric space $\langle X, d\rangle$ has a topological property $\mathcal{P}$, if the topological space $\langle X, \tau(d)\rangle$ has $\mathcal{P}$.   For a set $A\subseteq X$, let $\delta_d(A)=0$ if $A=\emptyset$, and let $\delta_d(A)=\sup\{d(x,y): x,y\in A\}$ if $A\neq \emptyset$. Then $\delta_d(A)$ is the \emph{diameter} of $A$ in $\langle X, d\rangle$.

\begin{definition}
\label{s1d1}
Let  $\mathbf{X}=\langle X, \tau\rangle$  be a topological space and let  $Y\subseteq X$. Suppose that $\mathcal{B}$ is a base of $\mathbf{X}$.
\begin{enumerate}
\item[(i)] The closure of $Y$ in $\mathbf{X}$ is denoted by $\cl_{\mathbf{X}}(Y)$.
\item[(ii)] $\tau|_Y=\{U\cap Y: U\in\tau\}$. $\mathbf{Y}=\langle Y, \tau|_Y\rangle$ is the subspace of $\mathbf{X}$ with the underlying set $Y$.
\end{enumerate}
\end{definition}

The set of all real numbers is denoted by $\mathbb{R}$ and, if it is not stated otherwise, $\mathbb{R}$ and every subspace of $\mathbb{R}$ are considered with their usual topology (denoted here by $\tau_{nat}$) and with the metric $d_{e}$ induced by the standard absolute value on $\mathbb{R}$.

 In the sequel, boldface letters will denote metric or topological spaces (called spaces in abbreviation) and lightface letters will denote their underlying sets. Metric and topological spaces will be called in brief \emph{spaces} if this is not misleading.

\begin{definition}
\label{s1d2} Let $\mathbf{X}$ be a space. Then:
\begin{enumerate}
 \item[(i)] $\mathbf{X}$ is \emph{first-countable} if every point of $X$ has a countable base of neighbourhoods;
 \item[(ii)] $\mathbf{X}$ is \emph{second-countable} if $\mathbf{X}$ has a countable base.
 \item[(iii)] $\mathbf{X}$ is \emph{compact} if every open cover of $\mathbf{X}$ has a finite subcover.
 \item[(iv)] $\mathbf{X}$ is \emph{separable} if it has a dense countable subset.
 \end{enumerate}
\end{definition}

Given a collection  $\{X_j: j\in J\}$ of sets, for every $i\in J$, we denote by $\pi_i$ the projection $\pi_i:\prod\limits_{j\in J}X_j\to X_i$ defined by $\pi_i(x)=x(i)$ for each $x\in\prod\limits_{j\in J}X_j$. If $\tau_j$ is a topology on $X_j$, then $\mathbf{X}=\prod\limits_{j\in J}\mathbf{X}_j$ denotes the Tychonoff product of the topological spaces $\mathbf{X}_j=\langle X_j, \tau_j\rangle$ with $j\in J$. If $\mathbf{X}_j=\mathbf{X}$ for every $j\in J$, then $\mathbf{X}^{J}=\prod\limits_{j\in J}\mathbf{X}_j$. As in \cite{En}, for an infinite set $J$ and the unit interval $[0,1]$ of $\mathbb{R}$, the cube $[0,1]^J$ is called the \emph{Tychonoff cube}. If $J$ is denumerable, then the Tychonoff cube $[0,1]^J$ is called the \emph{Hilbert cube}. In \cite{hh}, all Tychonoff cubes are called Hilbert cubes. In \cite{w}, Tychonoff cubes are called cubes. The Tychonoff cube $\mathbb{N}^{\omega}$, where $\mathbb{N}$ is the discrete subspace of positive integers of $\mathbb{R}$, is called the \emph{Baire space}.

We denote by $\mathbf{2}$ the discrete space $\langle 2, \mathcal{P}(2)\rangle$ where $2=\{0, 1\}$. Then, for every infinite set $J$, the space $\mathbf{2}^J$ is called a \emph{Cantor cube}. 

We recall that if $\prod\limits_{j\in J}X_j\neq\emptyset$, then it is said that the family $\{X_j: j\in J\}$ has a choice function, and every element of $\prod\limits_{j\in J}X_j$ is called a \emph{choice function} of the family $\{X_j: j\in J\}$. A \emph{multiple choice function} of $\{X_j: j\in J\}$ is every function $f\in\prod\limits_{j\in J}\mathcal{P}(X_j)$ such that, for every $j\in J$, $f(j)$ is a non-empty finite subset of $X_j$. A set $f$ is called a \emph{partial} (\emph{multiple}) \emph{choice function} of $\{X_j: j\in J\}$ if there exists an infinite subset $I$ of $J$ such that $f$ is a (multiple) choice function of $\{X_j: j\in I\}$. Given a non-indexed family $\mathcal{A}$, we treat $\mathcal{A}$ as an indexed family $\mathcal{A}=\{x: x\in\mathcal{A}\}$ to speak about a  (partial) choice function and a (partial) multiple choice function of $\mathcal{A}$.

\begin{definition}
\label{s1d3}
(Cf. \cite{br}, \cite{lo} and \cite{kerta}.) 
 A space $\mathbf{X}$ is said to be \emph{Loeb}  if the family of all non-empty closed subsets of $\mathbf{X}$ has a choice function.
\item[(ii)] If $\mathbf{X}$ is a Loeb space, then every  choice function of the family of all non-empty closed subsets of $\mathbf{X}$ is called a \emph{Loeb function} of $\mathbf{X}$.
\end{definition}

We recall that if $\{\mathbf{X}_{n}=\langle X_{n},d_{n}\rangle:n\in 
\mathbb{N}\}$ is a family of metric spaces, then, for $X=\prod\limits_{n\in \mathbb{N}}X_{n}$, the function $d:X\times
X\rightarrow \mathbb{R}$ given by: 
\begin{equation}
d(x,y)=\sum\limits_{n\in \mathbb{N}}\frac{\min\{d_n(x(n), y(n)), 1\}}{2^{n}}
\label{0}
\end{equation}%
for all $x,y\in X$, is
a metric on $X$ and the topology $\tau(d)$ in $X$ coincides with
the product topology of the family of spaces $\{\langle X_{n}, \tau(d_n)\rangle:n\in \mathbb{N%
}\}$ (see, e.g., \cite{w}). In the sequel, we shall always assume that whenever a family $\{\langle X_{n},d_{n}\rangle:n\in \mathbb{N}\}$ of metric spaces is given, then, the product $X=\prod\limits_{n\in \mathbb{N}}X_{n}$ carries the metric $d$
given by (\ref{0}).

Let  $\{X_j: j\in J\}$ be a disjoint family of sets, that is, $X_i\cap X_j=\emptyset$ for each pair $i,j$ of distinct elements of $J$. If $\tau_j$ is a topology on $X_j$ for every $j\in J$, then $\bigoplus\limits_{j\in J}\mathbf{X}_j$ denotes the direct sum of the spaces $\mathbf{X}_j=\langle X_j, \tau_j\rangle$ with $j\in J$.  Given a family $\{d_j: j\in J\}$ such that, for every $j\in J$, $d_j$ is a metric on $X_j$, one can define a metric $d$ on $X=\bigcup\limits_{j\in J}X_j$ as follows:
$$
(\ast)\text{ } d(x,y)=\begin{cases} 

1 &\text{ if there exist } i,j\in J \text{ such that } i\neq j,\\
& x\in X_i \text{ and } y\in X_j, \\ 
\min\{d_j(x,y), 1\}& \text{ if there exists } j\in J \text{ such that }  x,y\in X_j.
\end{cases}
$$

Then $\tau(d)$, where $d$ is defined by ($\ast$), coincides with the topology of the direct sum $\bigoplus\limits_{j\in j}\langle X_j, \tau(d_j)\rangle$, and the metric space $\langle \bigcup\limits_{j\in J}X_j, d\rangle$ is called the direct sum of the family $\{\langle X_j, d_j\rangle: j\in J\}$. In abbreviation, direct sums are called \emph{sums}.

Other topological notions used in this article but not defined here are standard. They can be found, for instance, in \cite{En} and \cite{w}. 

\subsection{The list of weaker forms of $\mathbf{AC}$}
\label{s1.3}

In this subsection, for the convenience of readers, we define and denote most of the weaker forms of $\mathbf{AC}$ used directly in this paper. If a form is not defined in the forthcoming sections, its definition can be found in this subsection. For the known forms given in \cite{hr}, \cite{hr1} or \cite{hh}, we quote in their statements the form number under which they are recorded in \cite{hr} (or in \cite{hr1} if they do not appear in \cite{hr}) and, if possible, we refer to their definitions in \cite{hh}. 

\begin{definition}
\label{s1d4}
\begin{enumerate}
\item $\mathbf{CAC}$ (\cite[Form 8]{hr}, \cite[Definition 2.5]{hh}): Every denumerable family of non-empty sets has a choice function.

\item  $\mathbf{CAC}(\mathbb{R})$ (\cite[Form 94]{hr}, \cite[Definition 2.9(1)]{hh}): Every denumerable family of non-empty subsets of $\mathbb{R}$ has a choice function.

\item $\mathbf{CAC}(\leq 2^{\aleph_0})$ (\cite[Form 16]{hr}):  For every family $\mathcal{A}=\{A_n: n\in\omega\}$ of non-empty sets such that, for every $n\in\omega$, $|A_n|\leq|\mathbb{R}|$, it holds that $\mathcal{A}$ has a choice function.

\item $\mathbf{CAC}_{\omega}$ (\cite[Form 32]{hr}): Every denumerable family of denumerable sets has a choice function.

\item $\mathbf{CAC}_{fin}$ (\cite[Form 10]{hr}, \cite[Definition 2.9(3)]{hh}): Every denumerable family of non-empty finite sets has a choice function.



\item $\mathbf{CMC}$ (\cite[Form 126]{hr}, \cite[Definition 2.10]{hh}): Every denumerable family of non-empty sets has a multiple choice function.

\item $\mathbf{CMC}_{\omega }$ (\cite[Form 350]{hr}): Every denumerable family of denumerable sets has a multiple choice function.

\item $\mathbf{CMC}(\leq 2^{\aleph_0})$: Every denumerable family of non-empty sets,
each of size $\leq |\mathbb{R}|$, has a multiple choice function.

\item $\mathbf{CUC}$ (\cite[Form 31]{hr}, \cite[Definition 3.2(1)]{hh}): Every countable union of
countable sets is countable.

\item $\mathbf{UT}(\aleph_0, \aleph_0, cuf)$ (\cite[Form 420]{hr1}): Every countable union of countable sets is a cuf set. (Cf. also \cite{hdhkr}.)

\item $\mathbf{IDI}$ (\cite[Form 9]{hr}): Every Dedekind-finite set is finite. (Equivalently, every infinite set is Dedekind-infinite.)

\item $\mathbf{AC}_{fin}$ (\cite[Form 62]{hr}): Every non-empty family of non-empty finite sets has a choice function.

\item $\mathbf{AC}_{WO}$ (\cite[Form 60]{hr}): Every non-empty family of non-empty well-orderable sets has e choice function.

\item $\mathbf{WOAC}_{fin}$ (\cite[Form 122]{hr}): Every non-empty well-orderable family of non-empty finite sets has a choice function.

\end{enumerate}
\end{definition}

\begin{remark}
\label{s1r5} 
The following are well-known facts in $\mathbf{ZF}$:
\begin{enumerate}
\item[(i)] $\mathbf{CAC}_{fin}$ is equivalent to each of the following sentences:
\begin{enumerate}
\item[(a)]  Every infinite well-ordered family of
non-empty finite sets has a partial choice function. (See \cite[Form \text{[10 O]}]{hr} and \cite[p. 23, Diagram 3.4]{hh}.)
\item[(b)] Every denumerable family of non-empty finite sets has a partial choice function. (See \cite[Form \text{[10 E]}]{hr}.)
\item[(c)] $\mathbf{CUC}_{fin}$: Every countable union of finite sets is countable. (See \cite[Definition 3.2(3)]{hh}.)
\end{enumerate}
 
\item[(ii)]  $\mathbf{CAC}$ is equivalent to the sentence: Every denumerable family of non-empty sets has a partial choice function. (See \cite[Form \text{[8 A]}]{hr}.)

\item[(iii)] $\mathbf{CMC}_{\omega}$ is equivalent to the following sentence: Every denumerable family of denumerable sets has a multiple choice function. 

\item[(iv)] The implications $\mathbf{CAC}(\leq 2^{\aleph_0})\rightarrow\mathbf{CUC}\rightarrow\mathbf{CAC}_{\omega}$ are true in every model of $\mathbf{ZF}$ (see \cite[p. 328]{hr}).  In Felgner's model $\mathcal{M}20$ in \cite{hr}, $\mathbf{CAC}_{\omega}$ is true and $\mathbf{CUC}$ is false. In Cohen's original model $\mathcal{M}1$ in \cite{hr}, $\mathbf{CUC}$ is true and $\mathbf{CAC}(\leq 2^{\aleph_0})$ is false. 
\end{enumerate}
\end{remark}

Let us pass to definitions of forms concerning metric and metrizable spaces.

\begin{definition}
\label{s1d7}
\begin{enumerate}
\item $\mathbf{CAC}(C,\mathbf{M}_{le})$: If $%
\{\langle X_{i},\tau_{i}\rangle:i\in \mathbb{N}\}$ is a family of non-empty compact metrizable spaces,
then the family $\{X_{i}:i\in \mathbb{N}\}$ has a choice function.

\item $\mathbf{CAC}(\mathbb{R},C)$: For every disjoint family $%
\mathcal{A}=\{A_{n}:n\in \mathbb{N}\}$ of non-empty subsets of $\mathbb{R}$,
if there exists a family $\{d_{n}:n\in \mathbb{N}\}$ of metrics such that, for every $n\in \mathbb{%
N},\langle A_{n},d_{n}\rangle$ is a compact metric space, then $\mathcal{A}$ has a choice function.


\item $\mathbf{CAC}(C,M)$: If $\{\langle X_n, d_n\rangle: n\in\omega\}$ is a family of non-empty compact metric spaces, then the family $\{X_n: n\in\omega\}$ has a choice function.

\item $\mathbf{CPM}(C, C)$: All countable products of compact metric spaces are compact. (Cf. \cite{kt}.)

\item $\mathbf{CPM}(C, S)$: All countable products of compact metric spaces are separable. (Cf. \cite{kt}.)

\item $\mathbf{CPM}_{le}$: Every countable product of metrizable spaces is
metrizable. 

\item $\mathbf{CSM}_{le}$ (Form 418 in \cite{hr1}): Every countable sum of
metrizable spaces is metrizable.

\item $\mathbf{M}(C,S)$: Every compact metrizable space is separable. (Cf. \cite{kk1}.)

\item $\mathbf{M}( C(\nleq\aleph_0), \geq 2^{\aleph_0})$: Every uncountable compact metrizable space is of size $\geq 2^{\aleph_0}$. 
\end{enumerate}
\end{definition}

The form  $\mathbf{M}( C(\nleq\aleph_0), \geq 2^{\aleph_0})$ is newly introduced here for its applications shown in Section \ref{s4}.

The forms from our next two definitions will be called  \emph{forms of type} $\mathbf{CPM}_{le}(\square, \square)$. They are defined in the spirit of \cite{kt}.

\begin{definition}
\label{s1d8}
Let $M, C, S, 2$ be the following properties: $M$--to be a metrizable space; $C$--to be a compact space; $S$--to be a separable space; $2$--to be a second-countable space. For properties $P,Q,R,T\in \{M, C, S, 2\}$, we define the following forms:
\begin{enumerate}
\item $\mathbf{CPM}_{le}(PQ,RT)$: Every countable product of
metrizable spaces, each having the properties $P$ and $Q$, has the properties $R$ and $T$.
\item $\mathbf{CPM}_{le}(P,T)$: Every countable product of metrizable spaces, each having the property $P$, has the property $T$.
\item $\mathbf{CPM}_{le}(PQ,T)$: Every countable product of metrizable spaces, each having the properties $P$ and $Q$, has the property $R$.
\item $\mathbf{CPM}_{le}(P, RT)$: Every countable product of metrizable spaces, each having the property $P$, has the properties $R$ and $T$.
\end{enumerate}
\end{definition}

Definition \ref{s1d8} is sufficient to get, for example, the following:

\begin{definition}
\label{s1d9}
\begin{enumerate}
\item  $\mathbf{CPM}_{le}(CS,MS)$: Every countable product of compact,
separable metrizable spaces is metrizable and separable.

\item $\mathbf{CPM}_{le}(C,MC)$: All countable products of compact
metrizable spaces are metrizable and compact.

\item $\mathbf{CPM}_{le}(CS,M)$: All countable products of compact,
separable, metrizable spaces are metrizable.

\item $\mathbf{CPM}_{le}(C,S)$: Every countable product of compact
metrizable spaces is separable.
\end{enumerate}
\end{definition}

Given a form $\mathbf{CPM}_{le}(\square, \square)$ concerning countable (denoted by $\mathbf{C}$) products (denoted by $\mathbf{P}$) of metrizable spaces (denoted by $\mathbf{M}$), if we replace $\mathbf{P}$ with $\mathbf{S}$, we obtain the form $\mathbf{CSM}_{le}(\square, \square)$ concerning countable sums of metrizable spaces. For example, using this rule, we can get the following:

\begin{definition}
\label{s1d10}
\begin{enumerate}
\item $\mathbf{CSM}_{le}(CS,MS)$: Every countable sum of compact, separable,
metrizable spaces is metrizable and separable.

\item $\mathbf{CSM}_{le}(C,MC)$: Every countable sum of compact metrizable
spaces is metrizable and compact;

\item $\mathbf{CSM}_{le}(CS,M)$: Every countable sum of compact, separable,
metrizable spaces is metrizable;

\item $\mathbf{CSM}_{le}(C,2)$: Every countable sum of compact metrizable
spaces is second countable.
\end{enumerate}
\end{definition}

There are differences between the notation in Definitions \ref{s1d8}-\ref{s1d10} and in \cite{kt}. For instance, if $P,T$ are  topological properties, then our $\mathbf{CPM}_{le}(P,MT)$  and $\mathbf{CPM}_{le}(P,T)$ from Definition \ref{s1d8} can be non-equivalent, while the  form $\mathbf{CPM}_{le}(P, T)$ in \cite{kt} coincides with the form $\mathbf{CPM}_{le}(P,MT)$ from our Definition \ref{s1d8}.

\subsection{The content of the article in brief}
\label{s1.4}

This article is devoted to the forms of type $\mathbf{CPM}_{le}(\square, \square)$ and $\mathbf{CSM}_{le}(\square, \square)$. Although, to a great extent, this work can be regarded as a continuation of \cite{kt}, many new results are included in the forthcoming sections. 

In Section \ref{s2}, we prove that the statements $\mathbf{CPM}_{le}(C,2)$, $\mathbf{CPM}_{le}(C,M2)$, $\mathbf{CPM}_{le}(C,MS)$, $\mathbf{CSM}_{le}(C,2)$, $\mathbf{CSM}%
_{le}(C,MS)$ and $\mathbf{CSM}_{le}(C,M2)$, as we have  expected, are all
equivalent statements in $\mathbf{ZF}$ (see Theorem \ref{s2t2}) and, in consequence, none of these
statements is a theorem of $\mathbf{ZF}$ (this follows immediately, for instance, from Theorem \ref{s1t16}). We show in Theorem \ref{s2t5} that $\mathbf{CAC}(\mathbb{R})$ implies that $\mathbf{CPM}_{le}(C,S)$ and $\mathbf{CPM}_{le}(C,2)$ are equivalent.  
In Theorem \ref{s2t12}, we show three implications; in particular,  we prove in $\mathbf{ZF}$ that the conjunction $\mathbf{CMC}(\leq 2^{\aleph_0})\wedge\mathbf{CPM}_{le}(C,C))$ implies $\mathbf{CPM}_{le}(C,2M)$, and $\mathbf{CPM}_{le}(C,2)$ implies $\mathbf{CPM}_{le}(C,C2)$.
We deduce that $\mathbf{CSM}_{le}(C,M)$ implies that  $\mathbf{CPM}_{le}(C,S)$, $\mathbf{CPM}_{le}(C,C)$, $\mathbf{CAC}(C,\mathbf{M}_{le})$ and $\mathbf{M}(C,S)$ are all equivalent in $\mathbf{ZF}$ (see Theorem \ref{s2t15}). 
Theorem \ref{s2t18} contains three distinct equivalents of $\mathbf{CUC}$ in $\mathbf{ZF}$; in particular, it shows that, in $\mathbf{ZF}$,  $\mathbf{CUC}$ is equivalent to the statement: ``Every countable product of compact, countable metrizable spaces is both compact and second-countable''. 

Theorem \ref{s2t17} is of special importance here. It contains five implications that are true in $\mathbf{ZF}$. The first two implications of Theorem \ref{s2t17} are the following: $\mathbf{CPM}_{le}(C,S)\rightarrow\mathbf{CUC}$ and $\mathbf{CPM}_{le}(CS,M)\rightarrow \mathbf{UT}(\aleph_0, \aleph_0, cuf)$.

In Section \ref{s3.1}, we remark that the conjunction $\mathbf{CUC}\wedge\neg\mathbf{CSM}_{le}\wedge\neg\mathbf{IDI}\wedge\neg\mathbf{CAC}(\leq 2^{\aleph_0})$ has a $\mathbf{ZF}$-model. In Section \ref{s3.2}, we prove that the permutation model constructed in \cite[proof of Theorem 14]{kt} is a $\mathbf{ZFA}$-model for the conjunction $\mathbf{CUC}\wedge\mathbf{IDI}\wedge\mathbf{WOAC}_{fin}\wedge\neg\mathbf{CSM}_{le}\wedge\neg\mathbf{CAC}(\leq 2^{\aleph_0})$, and we construct a new permutation model for the latter conjunction is Section \ref{s3.3}. 
Theorem \ref{s3t10} is the main result of Section \ref{s3.4}. It  asserts that the implications of Theorem \ref{s2t17} are not reversible in $\mathbf{ZF}$. 

In Section \ref{s4}, a new permutation model is constructed in which the forms $\mathbf{IDI}$, $\mathbf{WOAC}_{fin}$ and $\mathbf{M}(C(\not\leq\aleph_{0}),\geq 2^{\aleph_{0}})$ are all true, but the forms $\mathbf{CMC}_{\omega}$, $\mathbf{CAC}(C, \mathbf{M}_{le})$ and $\mathbf{CPM}_{le}(C,C)$ are all false (see Theorem \ref{t:main}).

A shortlist of open problems is included in Section \ref{s5}. Section \ref{s6} contains a diagram illustrating the implications established here. 

\subsection{A list of several known theorems}
\label{s1.5}

We list below some known theorems for future references.

\begin{theorem}
\label{s1t11}
$(\mathbf{ZF})$ 
\begin{enumerate}
\item[(i)] (Cf.  \cite{kt} and \cite[Section 4.7, Exercises E1-E2]{hh}.)  $\mathbf{CPM}_{le}$ and $\mathbf{CSM}_{le}$ are both equivalent to the following sentence: For any family $\{\langle X_{n}, \tau_{n}\rangle: n\in \mathbb{N}\}$ of metrizable spaces, there exists a
family of metrics $\{d_{n}: n\in \mathbb{N}\}$ such that, for every $n\in \mathbb{N%
},\tau(d_{n})=\tau_{n}$.

\item[(ii)] $\mathbf{CPM}_{le}\mathbf{(}C,M)$ and $\mathbf{CSM}_{le}\mathbf{(}C,M)$ are both equivalent to the following sentence: For every family $\{\langle X_{n}, \tau_{n}\rangle: n\in \mathbb{N}\}$ of compact metrizable spaces, there exists a
family of metrics $\{d_{n}: n\in \mathbb{N}\}$ such that, for every $n\in \mathbb{N%
},\tau(d_{n})=\tau_{n}$.
\end{enumerate}
\end{theorem}

\begin{theorem}
\label{s1t12}
(Cf. \cite[Theorem 2.1]{ew}.) $(\mathbf{ZF})$
If $J$ is a cuf set and $\{\langle X_j, d_j\rangle: j\in J\}$ is a family of metric spaces, then the product $\prod\limits_{j\in J}\langle X_j, \tau(d_j)\rangle$ is metrizable.
\end{theorem}

\begin{theorem}
\label{s1t13} 
(Cf. \cite[Theorem 2.2]{ew}.) $(\mathbf{ZF})$ Let $\mathbf{X}$ be a metrizable space consisting of at
least two points. Then, for a set $J$, the following conditions are
equivalent:
\begin{enumerate}
\item[(i)] $\mathbf{X}^{J}$ is metrizable;
\item[(ii)] $\mathbf{X}^{J}$ is first-countable;
\item[(iii)] $J$ is a cuf set.
\end{enumerate}
\end{theorem}

\begin{theorem}
\label{s1t14}
$(\mathbf{ZF})$ 
\begin{enumerate}
\item[(a)] (Cf. \cite{kert} and \cite{kt}.)  A compact metrizable space is Loeb iff it is
second-countable iff it is separable. In consequence, $\mathbf{M}(C,S)$ iff every compact metrizable space is Loeb iff every compact metrizable space is second-countable.
\item[(b)] (Cf. \cite[Theorem 8]{kt}.) The statements $\mathbf{M}(C,S)$, $\mathbf{CAC}(C,M)$,\newline $\mathbf{CPM}(C,S)$ and  $\mathbf{CPM}(C,C)$ are all equivalent.
\item[(c)] (Cf. \cite[Corollary 1 (a)]{kt}.) $\mathbf{CAC}(C,M)$ implies $\mathbf{CAC}_{fin}$.
\item[(d)] (Cf. \cite[Proposition 10(iii)]{ktw1}.) $\mathbf{CAC}_{fin}$ does not imply $\mathbf{M}(C,S)$.
\item[(e)] (Cf. \cite{kk}.) $\mathbf{CMC}$ implies  $\mathbf{CPM}_{le}$. In particular, $%
\mathbf{CMC}$ implies $\mathbf{CPM}_{le}(C,M)$.
\end{enumerate}
\end{theorem}

\begin{theorem}
\label{s1t15}$(\mathbf{ZF})$
\begin{enumerate}
\item[(i)] (Cf. \cite[Corollary 4.8]{gt}.)  $(\mathbf{ZF})$ (Urysohn's Metrization Theorem.) If $%
\mathbf{X}$ is a second-countable $T_3$-space, then $\mathbf{X}$ is metrizable.
\item[(ii)] (Cf. \cite[Theorem 2.1]{ktw2}.) If a $T_3$-space $\mathbf{X}$ has a cuf base, then $\mathbf{X}$ is metrizable.
\end{enumerate}
\end{theorem}

\begin{theorem}
\label{s1t16}
(Cf. \cite{kw1}.) 
$(\mathbf{ZF})$ Each of the following statements implies the one beneath it:
\begin{enumerate}
\item[(i)] $\mathbf{CMC}$.

\item[(ii)] Every countable product of one-point Hausdorff compactifications of denumerable discrete spaces is metrizable (first-countable).

\item[(iii)] $\mathbf{CMC}_{\omega }$.
\end{enumerate}
\end{theorem}


\begin{theorem}
\label{s1t18} (Cf. \cite{kw1}.) It holds in  $(\mathbf{ZF})$ that
 $\mathbf{UT}(\aleph_0, \aleph_0, cuf)$ is equivalent to the following sentence: Every countable product of one-point Hausdorff compactifications of denumerable discrete spaces is metrizable (equivalently, first-countable). 
\end{theorem}

\begin{theorem}
\label{s1t19} $(\mathbf{ZF})$ 
\begin{enumerate}
\item[(i)] (Cf. \cite{lo}.)  Let $\kappa $ be an infinite cardinal number of von Neumann, $\{\mathbf{X}_{i}:i\in \kappa \}$ be a family of
compact spaces, $\{f_{i}:i\in \kappa \}$ be a collection of functions such
that for every $i\in \kappa ,f_{i}$ is a Loeb function of $\mathbf{X}_{i}$. Then the Tychonoff product $%
\mathbf{X}=\prod\limits_{i\in \kappa }\mathbf{X}_{i}$ is compact. 

\item[(ii)] (Cf. \cite{kw}.) If $\mathbf{X}$ is a compact second-countable and metrizable
space, then $\mathbf{X}^{\omega }$ is compact and separable. In particular,
the Hilbert cube $[0,1]^{\mathbb{N}}$ is a compact, separable metrizable space.
\end{enumerate}
\end{theorem}

\section{Basic facts about forms concerning countable products of compact metrizable spaces}
\label{s2}
 
To prove, for instance, that a countable product of compact metrizable spaces is second-countable if and only if it is both metrizable and separable, we need the following proposition:

\begin{proposition}
\label{s2p1}
$(\mathbf{ZF})$ Let $\{\langle X_n, \tau_n\rangle: n\in\mathbb{N}\}$ be a family of topological spaces, let $X=\prod\limits_{n\in\mathbb{N}}X_n$ and $\mathbf{X}=\prod\limits_{n\in\mathbb{N}}\mathbf{X}_n$ where $\mathbf{X}_n=\langle X_n,\tau_n\rangle$ for every $n\in\mathbb{N}$. Then:
\begin{enumerate}
\item[(i)] if $\mathbf{X}$ is second-countable and $X\neq\emptyset$, then there exists a family $\{B_{n,m}: n,m\in\mathbb{N}\}$ such that, for every $n\in\mathbb{N}$, $\mathcal{B}_n=\{B_{n,m}: m\in\mathbb{N}\}$ is a countable base of $\mathbf{X}_n$;
\item[(ii)] if $\mathbf{X}$ is second-countable, $X\neq\emptyset$ and, for every $n\in\mathbb{N}$, $\mathbf{X}_n$ is a compact Hausdorff space, then there exist families $\{d_n: n\in\mathbb{N}\}$ and  $\{f_n: n\in\mathbb{N}\}$ such that, for every $n\in\mathbb{N}$, $f_n$ is a Loeb function of $\mathbf{X}_n$ and $d_n$ is a metric on $X_n$ such that $\tau(d_n)=\tau_n$.
\end{enumerate}
\end{proposition}
\begin{proof} (i) Let us assume that $\mathcal{B}=\{B_m: m\in\mathbb{N}\}$ is a countable base of $\mathbf{X}$ and $X\neq\emptyset$. We fix $y\in X$ and $k\in\mathbb{N}$. We define $\mathcal{B}_k=\{\pi_n[B_m]: m\in\mathbb{N}\}$ where $\pi_k: X\to X_k$ is the projection. Then $\mathcal{B}_m\subseteq \tau_k$.  To check that $\mathcal{B}_k$ is a base of $\mathbf{X}_k$, we take any non-empty $U\in\tau_k$ and $x\in U$. Let $z\in X$ be defined by: $z(k)=x$ and  $z(n)=y(n)$ if $n\in\mathbb{N}\setminus\{k\}$. Since $z\in\pi_k^{-1}[U]$, there exists $m\in\mathbb{N}$ such that $z\in B_m\subseteq\pi_k^{-1}(U)$. Then $x\in\pi_k[B_m]\subseteq U$.

(ii) Now, assume that, for every $n\in\mathbb{N}$, $\mathbf{X}_n$ is a compact Hausdorff space, $\mathbf{X}$ is second-countable and $X\neq\emptyset$. Since every compact Hausdorff space is regular and products of $T_3$-spaces are $T_3$-spaces, the space $\mathbf{X}$ is a $T_3$-space. Hence, by Theorem \ref{s1t15}, $\mathbf{X}$ is metrizable. Fixing a metric $d$ which induces the topology of $\mathbf{X}$, we can easily define a family $\{d_n: n\in\mathbb{N}\}$ such that, for every $n\in\mathbb{N}$, $d_n$ is a metric on $X_n$ such that $\tau(d_n)=\tau_n$. Furthermore, it follows from (i) that there exists a family $\{B_{n,m}: n,m\in\mathbb{N}\}$ such that, for every $n\in\mathbb{N}$,  $\mathcal{B}_n=\{B_{n,m}: m\in\mathbb{N}\}$ is a base of $\mathbf{X}_n$. Now, in much the same way, as in  \cite[proof of $(iv)\rightarrow(i) $ of Theorem 2.1]{kert}, we can effectively define a family $\{f_n: n\in\mathbb{N}\}$ such that, for every $n\in\mathbb{N}$, $f_n$ is a Loeb function of $\mathbf{X}_n$.
\end{proof}

\begin{theorem}
\label{s2t2}
$(\mathbf{ZF})$
\begin{enumerate}
\item[(i)] The following sentences are all equivalent:\newline
 $\mathbf{CPM}_{le}(C,2)$, $%
\mathbf{CPM}_{le}(C,MS)$, $\mathbf{CPM}_{le}(C,M2)$, $\mathbf{CPM}_{le}(C, C2)$,\newline
 $\mathbf{CSM}_{le}(C,2)$, $\mathbf{CSM}_{le}(C,MS)$, $\mathbf{CSM}_{le}(C,M2)$.

\item[(ii)] $\mathbf{CPM}_{le}(C,2)\rightarrow\mathbf{CPM}_{le}(C,S)$.

\item[(iii)] $\mathbf{CPM}_{le}(C,S)\leftrightarrow \mathbf{CSM}_{le}(C, S)$.

\item[(iv)]  $\mathbf{CPM}_{le}(C,M2)\rightarrow \mathbf{CPM}_{le}(C,MC)\rightarrow 
\mathbf{CPM}_{le}(C,C)$.

\item[(v)] $(\mathbf{CPM}_{le}(C,M)\wedge\mathbf{CPM}(C,C))\leftrightarrow \mathbf{CPM}_{le}(C, MC)$.

\item[(vi)]  $\mathbf{CPM}_{le}(C, MC)\rightarrow \mathbf{CPM}_{le}(C,S)$.

\item[(vii)]  $(\mathbf{CPM}_{le}(C, MC)\wedge\mathbf{M}(C,S))\leftrightarrow\mathbf{CPM}_{le}(C, MS)$.

\item[(viii)] The following sentences are all equivalent:\newline
$\mathbf{CPM}_{le}(CS,MS)$, $\mathbf{CPM}_{le}(CS,2)$, $\mathbf{CPM}%
_{le}(CS,MC)$,\newline $\mathbf{CPM}_{le}(CS,M2)$, $\mathbf{CPM}_{le}(C2,MS)$,
$\mathbf{CPM}_{le}(C2,MC)$,\newline $\mathbf{CPM}_{le}(C2,M2)$, $\mathbf{CSM}%
_{le}(CS,MS)$, $\mathbf{CSM}_{le}(CS,M2)$,\newline $\mathbf{CSM}_{le}(C2,MS)$, $\mathbf{CSM}_{le}(C2,M2)$.

\item[(ix)] Each of the forms listed in (i) implies each of the forms listed in (viii).
\end{enumerate}
\end{theorem}

\begin{proof}
 Let $\{\mathbf{X}_n: n\in\omega\}$ be a family of non-empty compact metrizable spaces and let $\mathbf{X}=\prod\limits_{n\in\omega}\mathbf{X}_n$. Without loss of generality, we may assume that $X_n\cap X_m=\emptyset$ for each pair $n,m$ of distinct natural numbers. Let $(\infty_n)_{n\in\omega}$ be a sequence of elements such that, for each pair $n,m$ of distinct members of $\omega$, the sets $X_n\cup\{\infty_n\}$ and $X_m\cup\{\infty_m\}$ are disjoint and $\infty_n\notin X_n$. For $n\in\omega$, we put $Y_n=X_n\cup\{\infty_n\}$ and $\mathbf{Y}_n=\mathbf{X}_n\oplus\{\infty_n\}$ where $\{\infty_n\}$ is treated as the one-element discrete space. Let $\mathbf{Y}=\prod\limits_{n\in\omega}\mathbf{Y}_n$, $Z=\bigcup\limits_{n\in\omega}Y_n$ and  $\mathbf{Z}=\bigoplus\limits_{n\in\omega}\mathbf{Y}_n$. 
 
(i) If $\mathbf{X}$ is the empty space, it is separable, second-countable and metrizable. Suppose that $X\neq\emptyset$ and $\mathbf{X}$ is second-countable. It follows from Proposition \ref{s2p1}, taken together with Theorems \ref{s1t19}(i) and \ref{s1t12}, that $\mathbf{X}$ is a compact metrizable space. By Theorem \ref{s1t14}($a$), $\mathbf{X}$ is a Loeb, separable space.  All this taken together with the fact that every separable metrizable space is second-countable shows that  $\mathbf{CPM}_{le}(C,2)$, $\mathbf{CPM}_{le}(C,C2)$, $\mathbf{CPM}_{le}(C,MS)$ and $\mathbf{CPM}_{le}(C,M2)$ are all equivalent.

Now, let us notice that if there exists a family $\{B_{n,m}: n,m\in\omega\}$ such that, for every $n\in\omega$, $\{B_{n,m}: m\in\omega\}$ is a countable base of $\mathbf{Y}_n$, then both $\mathbf{Y}$ and $\mathbf{Z}$ are second-countable. Hence, it follows from Proposition \ref{s2p1} and properties of direct sums of topological spaces that if $\mathbf{Y}$ or $\mathbf{Z}$ is second-countable, then both $\mathbf{Y}$ and $\mathbf{Z}$ are second-countable. Clearly, $\mathbf{Z}$ is second-countable if and only if $\bigoplus\limits_{n\in\omega}\mathbf{X}_n$ is second-countable and, moreover, if $\mathbf{Y}$ is second-countable, then so is $\mathbf{X}$. In consequence, $\mathbf{CPM}_{le}(C, 2)\leftrightarrow\mathbf{CSM}_{le}(C, 2)$. Hence, by Proposition \ref{s2p1}(ii),  $\mathbf{CSM}_{le}(C,2)\leftrightarrow\mathbf{CSM}_{le}(C,M2)$. 

We notice that if $\bigoplus\limits_{n\in\omega}\mathbf{X}_n$ is both metrizable and second-countable, then, similarly to the proof of Proposition \ref{s2p1}, one can show that there exist families $\{B_{n,m}: n,m\in\omega\}$ and $\{f_n: n\in\omega\}$ such that, for every $n\in\omega$, $\{B_{n,m}: m\in\omega\}$ is a countable base of $\mathbf{X}_n$ and $f_n$ is a Loeb function of $\mathbf{X}_n$. Then the set $D=\{f_n(\cl_{\mathbf{X}_n}(B_{n,m})): n,m\in\omega\}$ is countable and dense in $\bigoplus_{n\in\omega}\mathbf{X}_n$. Hence $\mathbf{CSM}(C, M2)$ implies $\mathbf{CSM}_{le}(C, MS)$. 

To conclude the proof of (i), let us notice that if $\bigoplus\limits_{n\in\omega}\mathbf{X}_n$ is both metrizable and separable, $d$ is a metric which induces the topology of $\bigoplus\limits_{n\in\omega}\mathbf{X}_n$  and $A$ is a dense set in $\bigoplus\limits_{n\in\omega}\mathbf{X}_n$, then the family $\mathcal{B}=\{ \{y\in X_n: d(x,y)<\frac{1}{m+1}\}: x\in A\cap X_n\text{ and } n,m\in\omega\}$ is a countable base of $\bigoplus\limits_{n\in\omega}\mathbf{X}_n$. Hence $\mathbf{CSM}_{le}(C, MS)$ implies $\mathbf{CSM}_{le}(C, M2)$. \medskip

(ii) It follows from (i) that $\mathbf{CPM}_{le}(C,2)$ implies $\mathbf{CPM}_{le}(C,MS)$. Hence (ii) holds. \medskip

(iii) Suppose that $\mathbf{Y}$ is separable and $H=\{y_m: m\in\mathbb{N}\}$ is a countable dense set in $\mathbf{Y}$. Let $E=\{ y_m(n): n,m\in\omega\}$. Then $E$ is a countable dense subset of $\mathbf{Z}$. This shows that $\mathbf{CPM}_{le}(C, S)$ implies $\mathbf{CSM}_{le}(C, S)$. On the other hand, assuming that $G=\{x_m: m\in\mathbb{N}\}$ is a countable dense set of $\bigoplus_{n\in\omega}\mathbf{X}_n$, for every $n\in\omega$, we put $M(n)=\{m\in\omega: x_m\in X_n\}$, $m(n)=\min M(n)$ and $L_n=\{x_m: m\in M(n)\}$. For each $k\in\omega$. let $P_k$ be the set of all points $x$  of $X=\prod\limits_{n\in\omega}X_n$ such that: $x(n)\in L(n)$ for every $n\in k+1$, and $x(n)=x_{m(n)}$ for every $n\in\omega\setminus (k+1)$. Then the set $P=\bigcup_{k\in\omega}P_k$ is countable and dense in $\mathbf{X}$. Hence $\mathbf{CSM}_{le}(C, S)$ implies $\mathbf{CPM}_{le}(C,S)$.\medskip

(iv) That $\mathbf{CPM}_{le}(C, M2)$ implies $\mathbf{CPM}_{le}(C, MC)$ was shown in the proof of (i). It is trivial that $\mathbf{CPM}_{le}(C, MC)$ implies $\mathbf{CPM}_{le}(C,C)$. 

It is obvious that (v) holds.\medskip

(vi)  Assuming  $\mathbf{CPM}_{le}(C, MC)$, we show that the space $\mathbf{X}$ is separable.  By $\mathbf{CPM}_{le}(C, MC)$, the space $\mathbf{X}$ is compact and metrizable. Clearly, $\mathbf{CPM}_{le}(C, MC)$ implies $\mathbf{CPM}(C, MC)$. In the light of Theorem \ref{s1t14}($b$), $\mathbf{CPM}(C, MC)$ and $\mathbf{M}(C, S)$ are equivalent. Hence $\mathbf{X}$ is separable and, in consequence, (vi) holds.\medskip

That (vii) holds can be deduced from Theorem \ref{s1t14}($b$).

(viii) It is obvious that $\mathbf{CPM}_{le}(C,MS)$ implies $\mathbf{CPM}_{le}(CS,2)$. On the other hand,  $\mathbf{CPM}_{le}(CS,2)$ implies $\mathbf{CPM}_{le}(CS, M2)$ by Theorem \ref{s1t15}. It follows from the proof of (i) and from Theorem \ref{s1t14} ($b$) that (viii) also holds. Now, it is obvious that (ix) is true.
\end{proof}

\begin{corollary}
\label{s2c3}
$(\mathbf{ZF})$ For every family $\{\mathbf{X}: n\in\omega\}$ of compact metrizable spaces, it holds that $\mathbf{X}=\prod_{n\in\omega}\mathbf{X}_n$ is second-countable if and only if $\mathbf{X}$ is both metrizable and separable. 
\end{corollary}
\begin{proof}
This follows from the arguments given in the proof of (i) of Theorem \ref{s2t2}.
\end{proof}

Since every second-countable compact Hausdorff space is metrizable by Theorem \ref{s1t15}, we can write down the following corollary to Theorem \ref{s2t2}(i):

\begin{corollary}
\label{s2c13}
$(\mathbf{ZF})$  $\mathbf{CPM}_{le}(C,2)$ implies that every countable product of compact Hausdorff second-countable spaces is a metrizable, compact second-countable space. 
\end{corollary}

The following theorem shows that the implication from (ii) of Theorem \ref{s2t2} is reversible in $\mathbf{ZF}+\mathbf{CAC}(\mathbb{R})$.

\begin{theorem}
\label{s2t5}
$(\mathbf{ZF})$
$\mathbf{CAC}(\mathbb{R})$ implies that  $\mathbf{CPM}_{le}(C,S)$ and $\mathbf{CPM}%
_{le}(C,2)$ are equivalent.
\end{theorem}

\begin{proof}
Let us assume both $\mathbf{CAC}(\mathbb{R})$ and $\mathbf{CPM}_{le}(C,S)$.
Fix a family $\{ \mathbf{X}_{n}:n\in \omega\}$ of compact metrizable spaces and let $\mathbf{X}=\prod\limits_{n\in \omega%
}\mathbf{X}_{n}$. Let $\mathbf{X}_n=\langle X_n, \tau_n\rangle$ for $n\in\omega$. We prove that $\mathbf{X}$ is second-countable. Since the empty space is second-countable, we may assume that $X\neq\emptyset$. By our hypothesis, $\mathbf{X}$ is separable. Fix a dense subset $%
D=\{x_{i}:i\in \omega\}$ of $\mathbf{X}$. Clearly, for every $n\in 
\omega$, the set $D_{n}=\{x_{i}(n):i\in \omega\}$ is dense in $\mathbf{X}_{n}$.
For every $n\in \omega$, let $F_{n}$ be the set of all functions $d\in \mathbb{R}^{D_{n}\times
D_{n}}$ such that there exists a metric $\rho$ on $X_n$ such that $\tau(\rho)=\tau_n$ and $d(x,y)=\rho(x,y)$ for all $x,y\in D_n$. Since there exists a family $\{f_n: n\in\omega\}$ of injections $f_n: D_n\times D_n\to\omega\times\omega$ and  $|\mathbb{R}^{D_{n}\times
D_{n}}|=|\mathbb{R}^{\omega \times \omega }|=|\mathbb{R}^{\omega }|=|\mathbb{%
R}|$, there exists also a family $\{g_n: n\in\omega\}$ of injections $g_n: F_n\to\mathbb{R}$. By $\mathbf{CAC}(\mathbb{R})$, we can fix $\psi\in\prod\limits_{n\in\omega}g_n[F_n]$.  For every $n\in\omega$, let $\sigma_n=g_n^{-1}(\psi(n))$.  Since $\sigma_n\in F_n$ and $%
D_{n}\times D_{n}$ is dense in $\mathbf{X}_{n}\times \mathbf{X}_{n}$, there exists a unique metric $d_{n}$ on $X_n$ such that $\tau(d_n)=\tau_n$ and $\sigma_n(x,y)=d_n(x,y)$ for all $x,y\in D_n$. Therefore, $\mathbf{X}$ is
metrizable and, in consequence, second-countable. This, together with Theorem \ref{s2t2}(ii), completes the proof.
\end{proof}

From the proof of Theorem \ref{s2t5}, we deduce the following corollary:

\begin{corollary}
\label{s2c6}
Let $\mathcal{M}$ be any model of $\mathbf{ZF}+\mathbf{CAC}(\mathbb{R})$. In $\mathcal{M}$, let $\{\mathbf{X}_n: n\in\omega\}$ be a family of compact metrizable spaces. Then it holds in $\mathcal{M}$ that $\mathbf{X}=\prod_{n\in\omega}\mathbf{X}_n$ is second-countable if and only if $\mathbf{X}$ is separable. 
\end{corollary}

In view of Theorem \ref{s2t2}, one may ask the following questions:

\begin{question}
\label{s2q10}
\begin{enumerate}
\item[(i)] Does $\mathbf{CPM}_{le}\mathbf{(}C,M)$ imply $\mathbf{CPM}_{le}(C,S)$ in $\mathbf{ZF}$?
\item[(ii)]  Does $\mathbf{CPM}_{le}(C,C)$ imply $\mathbf{CPM}_{le}(C,M)$ in $\mathbf{ZF}$?
\item[(iii)] Does $\mathbf{CPM}_{le}(CS,MS)$ imply $\mathbf{M}(C,S)$ in $\mathbf{ZF}$?
\end{enumerate}
\end{question}

\begin{remark}
\label{s2r11}
With regard to Question \ref{s2q10}(i), we notice
 that $\mathbf{CPM}_{le}\mathbf{(}C,M)$ implies neither $%
\mathbf{CPM(}C,C)$  nor $\mathbf{CPM}_{le}\mathbf{(}C,S)$ in $\mathbf{ZFA}$. Indeed, by Theorem \ref
{s1t14}($b$), it holds in $\mathbf{ZF}$ that $\mathbf{CPM}(C,S)$ and $\mathbf{CPM}(C,C)$ are equivalent. This equivalence is also valid in $\mathbf{ZFA}$. Moreover, since Theorem \ref{s1t14} ($b$)--($c$) is valid in $\mathbf{ZFA}$, the following conjunction is true in $\mathbf{ZFA}$:
$$(\mathbf{CMC}\rightarrow \mathbf{CPM}_{le}(C,M))\wedge (\mathbf{CPM}(C,C)\rightarrow \mathbf{CAC}_{fin}).$$
Hence, in every model of $\mathbf{ZFA}$ satisfying $\mathbf{CMC}$ and
the negation of $\mathbf{CAC}_{fin}$, for example, in the Second Fraenkel Model $\mathcal{N}$2 of \cite{hr}, $\mathbf{CPM}_{le}(C,M)$ is true but both $\mathbf{CPM}(C,C)$ and $\mathbf{CPM}_{le}(C,S)$ are false.
\end{remark}

Even though we do not know the full answer to Question \ref{s2q10} (ii), we give the
following partial answer.

\begin{theorem}
\label{s2t12}
$(\mathbf{ZF})$
\begin{enumerate}
\item[(i)] $\mathbf{CPM}_{le}(C,C)\rightarrow\mathbf{CAC}(C, \mathbf{M}_{le})\rightarrow\mathbf{M}(C,S)$.
\item[(ii)] $(\mathbf{CMC}(\leq 2^{\aleph_0})\wedge\mathbf{CPM}_{le}(C,C))\rightarrow\mathbf{CPM}_{le}(C,2M)$. 
\end{enumerate}
\end{theorem}

\begin{proof}
 That the first implication of (i) (``$\mathbf{CPM}_{le}(C,C)\rightarrow\mathbf{CAC}(C,\mathbf{M}_{le})$'') is true can be shown similarly to the proof in \cite{kl} that the Tychonoff Product Theorem implies $\mathbf{AC}$. That the second implication of (i) ( ``$\mathbf{CAC}(C,\mathbf{M}_{le})\rightarrow\mathbf{M}(C,S)$'') is true follows from Theorem \ref{s1t14} ($b$).\medskip
 
 (ii) Assuming $\mathbf{CMC}(\leq 2^{\aleph_0})$ and $\mathbf{CPM}_{le}(C,C)$, we prove $\mathbf{CPM}_{le}(C,2M)$. To this aim, we fix a family  $\{\mathbf{X}_{n}: n\in\omega\}$ of compact metrizable spaces $\mathbf{X}_n=\langle X_n, \tau_n\rangle$, and put $\mathbf{X}=\prod\limits_{n\in \omega}\mathbf{X}_{n}$.  We show that $\mathbf{X}$ is separable and metrizable, hence second-countable also. We may assume that $X=\prod_{n\in\omega}X_n\neq\emptyset$. By (i), $\mathbf{CAC}(C,\mathbf{M}_{le})$ and $\mathbf{M}(C,S)$ hold. Thus,
for every $n\in \omega$, $\mathbf{X}_{n}$ is separable. For every $n\in\omega$, let 
\[
F_{n}=\{d\in \mathbb{R}^{X_{n}\times X_{n}}:d\text{ is a metric on }X_{n}%
\text{ such that }\tau_{n}=\tau(d_n)\}.
\]%
For a given $n\in\omega$, let us prove that the non-empty set $F_{n}$ is equipotent to a subset of $\mathbb{R}$. Since $\mathbf{X}_{n}$ is separable, the product $\mathbf{X}_{n}\times \mathbf{X}_{n}$ is also separable. Let $G$ be a countable dense set in $\mathbf{X}_{n}\times \mathbf{X}_{n}$. Since the
members of $F_{n}$ are continuous functions, their values are determined on $%
G$, i.e., if $d_{1},d_{2}\in F_{n}$ and the restrictions $d_{1}\upharpoonright G$ and $%
d_{2}\upharpoonright G$ are equal, then $d_{1}=d_{2}$. Therefore, $F_n$ is equipotent to a subset of $\mathbb{R}^{G}$. However, $\mathbb{R}^{G}$ and $\mathbb{R}$ are equipotent. This is why $F_n$ is equipotent to a subset of $\mathbb{R}$.

Fix, by $\mathbf{CMC}(\leq 2^{\aleph_0})$, a multiple choice function $f$ of $\{F_{n}:n\in \omega\}$. For every $n\in \omega$, we
define a function $d_{n}:X_n\times X_n\to\mathbb{R}$ by requiring: 
\[
d_{n}(x,y)=\max\{d(x,y): d\in f(n)\}\text{ where } x,y\in X_n.
\]%
It is straightforward to verify that for every $n\in \omega$, $d_{n}$ is
a metric on $X_{n}$ such that $\tau_n=\tau(d_n)$. It follows from Theorem \ref{s1t12} that $\mathbf{X}$ is metrizable. By $\mathbf{CPM}_{le}(C,C)$, the space $\mathbf{X}$ is compact. Hence,
by $\mathbf{M}(C,S)$, $\mathbf{X}$ is separable. This completes the proof of (ii) because every separable metrizable space is second-countable.\medskip
\end{proof}

Theorem \ref{s1t14} ($a$)-($b$) is a characterization of $\mathbf{M}(C,S)$. The following question pops up at this point:

\begin{question}
\label{s2q14}
Are the sentences $\mathbf{M}(C,S)$, $\mathbf{CAC}(C, \mathbf{M}_{le})$, $\mathbf{CPM}_{le}(C,S)$ and $\mathbf{CPM}_{le}(C,C)$ equivalent in $\mathbf{ZF}$?
\end{question}

We do not know the complete answer to this question, but, in the forthcoming Theorem \ref{s2t15}, we show that the answer is in the affirmative in $\mathbf{ZF}+\mathbf{CSM}_{le}(C,M)$.

\begin{theorem}
\label{s2t15}
$(\mathbf{ZF})$ $\mathbf{CSM}_{le}(C,M)$ implies that the following sentences are all equivalent: $\mathbf{CPM}_{le}(C,S)$, $\mathbf{CPM}_{le}(C,C)$, $\mathbf{CAC}(C,\mathbf{M}_{le})$ and $\mathbf{M}(C,S)$.
\end{theorem}

\begin{proof}
Assume  $\mathbf{CSM}_{le}(C,M)$. Then, by Theorems \ref{s1t11}(ii), \ref{s2t2}(i) and \ref{s2t12}(i), the following implications are true:
 $$\mathbf{CPM}_{le}(C,S)\rightarrow\mathbf{CPM}_{le}(C,MS)\rightarrow\mathbf{CPM}_{le}(C,C).$$
Furthermore, by Theorem \ref{s1t14}($b$), $\mathbf{M}(C,S)$ implies $\mathbf{CPM}_{le}(C,S)$. This, together whith Theorem \ref{s2t12}(i), completes the proof.
\end{proof}

\begin{definition}
\label{s2d16}
For an infinite set $A$ and an element $\infty\notin A$, let $A(\infty)=A\cup\{\infty\}$, 
$$\tau=\mathcal{P}(A)\cup\{ A(\infty)\setminus F: F\in[A]^{<\omega}\}$$
and $\mathbf{A}(\infty)=\langle A(\infty), \tau\rangle$.
\end{definition}

We notice that, for every infinite set $A$, the space $\mathbf{A}(\infty)$ is the unique (up to equivalence) one-point Hausdorff compactification of the discrete space $\langle A, \mathcal{P}(A)\rangle$.

\begin{theorem}
\label{s2t17} 
$(\mathbf{ZF})$
\begin{enumerate}
\item[(i)] $\mathbf{CPM}_{le}(C,S)\rightarrow\mathbf{CUC}$.

\item[(ii)] $\mathbf{CPM}_{le}(CS,M)\rightarrow \mathbf{UT}(\aleph_0, \aleph_0, cuf)\rightarrow \mathbf{CMC}_{\omega}$. 

\item[(iii)]   $(\mathbf{CPM}_{le}(CS,M)\wedge\mathbf{CUC}_{fin})\rightarrow \mathbf{CUC}$.

\item[(iv)] $\mathbf{CPM}_{le}(CS,MS)\rightarrow\mathbf{CUC}$.
\end{enumerate}
\end{theorem}

\begin{proof}
(i) Assuming $\mathbf{CPM}_{le}(C,S)$, we fix a disjoint family $\mathcal{A}=\{A_{n}:n\in\omega\}$ 
of countable sets. Take any element $\infty\notin\bigcup\mathcal{A}$. For every $n\in \omega$, let $X_n=A_n\cup\{\infty\}$. If $A_n$ is finite, let $\mathbf{X}_{n}=\langle X_n, \mathcal{P}(X_n)\rangle$. If $A_n$ is infinite, let $\mathbf{X}_n=\mathbf{A}_n(\infty)$ (see Definition \ref{s2d16}). Since, for every $n\in\omega$, $\mathbf{X}_n$ is compact and metrizable, it follows from $\mathbf{CPM}_{le}(C,S)$ that the product $\mathbf{X}=\prod\limits_{n\in \omega}\mathbf{X}_{n}$ is
separable. Fix a countable dense subset $D=\{y_n: n\in\omega\}$ of $\mathbf{X}$. For every $n\in \omega$ we define a  function $%
f_{n}:A_{n}\to \omega$ as follows: for every $a\in A_{n}$,%
\[
f_{n}(a)=\min \{i\in \omega: y_{i}\in \pi _{n}^{-1}(a)\}
\]%
\noindent where $\pi_n:\mathbf{X}\to \mathbf{X}_n$ is the projection. Since, for every pair $a,b$ of distinct elements of $A_n$, $\pi _{n}^{-1}(a)\cap \pi
_{n}^{-1}(b)=\emptyset$, it follows that $f_{n}(a)\neq f_{n}(b)$. Therefore, for every $n\in\omega$, the function $f_{n}$ is an injection. We define a function $h:\bigcup \mathcal{A}%
\rightarrow \omega\times \omega$ as follows:  for every $a\in\bigcup\mathcal{A}$, $%
h(a)=\langle n_{a},f_{n_{a}}(a)\rangle$ where $n_{a}$ is the unique $n\in \omega$
with $a\in A_{n}.$ Since $h$ is an injection, it follows that $\bigcup \mathcal{A}$ is equipotent to a subset of $\omega\times\omega$,  Hence, $\bigcup \mathcal{A}$ is countable as required.\medskip

(ii) We assume $\mathbf{CPM}_{le}(CS,M)$ and prove $\mathbf{UT}(\aleph_0, \aleph_0, cuf)$. To this aim, we fix a disjoint family $\mathcal{J}=\{J_{n}:n\in \mathbb{N}\}$ of non-empty countable sets. For every $n\in \mathbb{N}$, let $\mathbf{%
X}_{n}$ denote the Hilbert cube $[0,1]^{J_{n}}$. Then, for every $n\in\mathbb{N}$, $\mathbf{X}_n$ is homeomorphic with a closed subspace of the Hilbert cube $[0,1]^{\mathbb{N}}$, so $\mathbf{X}_n$ is a compact, separable metrizable space.  By our hypothesis, the space $\mathbf{X}=\prod\limits_{n\in\mathbb{N}}\mathbf{X}_n$ is metrizable. Let $J=\bigcup\limits_{n\in\mathbb{N}}J_n$. Since $\mathbf{X}$ is homeomorphic with the Tychonoff cube $[0,1]^J$, it follows that $[0,1]^J$ is metrizable. By Theorem \ref{s1t13}, $J$ is a cuf set. This completes the proof of the first implication of (ii). The second implication of (ii) follows from Theorems \ref{s1t18} and \ref{s1t16}. \medskip

(iii) That (iii) holds is a simple consequence of the first implication of (ii). \medskip

(iv) Assume $\mathbf{CPM}_{le}(CS,MS)$. Then $\mathbf{CSM}_{le}(CS,MS)$ holds by Theorem \ref{s2t2}(viii). This trivially implies $\mathbf{CAC}_{fin}$. It has been mentioned in Remark \ref{s1r5}(i) that  $\mathbf{CUC}_{fin}$ and $\mathbf{CAC}_{fin}$ are equivalent. Hence $\mathbf{CUC}_{fin}$ holds. Clearly, $\mathbf{CPM}_{le}(CS, M)$ also holds. This, together with (iii), implies (iv) 
\end{proof}

\begin{theorem}
\label{s2t18}
$(\mathbf{ZF})$ $\mathbf{CUC}$ is equivalent to each of the following (i)--(iii):
\begin{enumerate}
\item[(i)]  all countable products of countable topological spaces are separable;
\item[(ii)] every countable sum of compact countable metrizable spaces is second-countable;
\item[(iii)] every countable product of compact countable metrizable spaces is compact and second-countable.
\end{enumerate}
\end{theorem}

\begin{proof}
 To show that $\mathbf{CUC}$ implies (i)--(iii), we assume $\mathbf{CUC}$ and fix a family $\{\mathbf{X}_{n}:n\in \omega\}$ of
countable topological spaces $\mathbf{X}_n=\langle X_n, \tau_n\rangle$. Without loss of generality, we may assume that $X_n\cap X_m=\emptyset$ for all $m,n\in\omega$ with $m\neq n$. Let $X=\prod\limits_{n\in\omega}X_{n}$ and $\mathbf{X}=\prod\limits_{n\in\omega}\mathbf{X}_n$. \medskip

($\mathbf{CUC}\rightarrow (i)$) If $X=\emptyset$, the space $\mathbf{X}$ is separable. Suppose that $X\neq\emptyset$ and fix $x_0\in X$. For every $n\in\omega$, let $E_n=\{ x\in X: (\forall i\in\omega\setminus (n+1)) x(i)=x_0(i)\}$.  By $\mathbf{CUC}$, the set $E=\bigcup\limits_{n\in\omega}E_n$ is countable. Clearly, $E$ is dense in $\mathbf{X}$. This shows that $\mathbf{CUC}$ implies (i)\medskip

($\mathbf{CUC}\rightarrow ((ii)\wedge(iii))$)  Let $\mathbf{Y}=\bigoplus\limits_{n\in\omega}\mathbf{X}_n$ and assume that, for every $n\in\omega$, the space $\mathbf{X}_n$ is compact, countable and metrizable. Then, for every $n\in\omega$, the space $\mathbf{X}_n$ is scattered, that is, every subspace of $\mathbf{X}_n$ has an isolated point. Mimicking  the arguments given in \cite[proof of Theorem 8]{kft}, one can show that it follows from $\mathbf{CUC}$ that there exists a family $\{\mathcal{B}_n: n\in\omega\}$ such that, for every $n\in\omega$, $\mathcal{B}_n$ is a countable base of $\mathbf{X}_n$. By $\mathbf{CUC}$, the family $\mathcal{B}=\bigcup\limits_{n\in\omega}\mathcal{B}_n$ is countable, hence $\mathbf{Y}$ is second-countable. Moreover, there exists a family $\{B_{n,m}: n,m\in\omega\}$ such that, for every $n\in\omega$, the family $\{B_{n,m}: m\in\omega\}$ is a base of $\mathbf{X}_n$. In much the same way, as in the proof of Proposition \ref{s2p1}, we can show that there is a collection $\{d_n: n\in\omega\}$ such that, for every $n\in\omega$, $d_n$ is a metric on $X_n$ such that $\tau(d_n)=\tau_n$. This implies that both $\mathbf{Y}$ and $\mathbf{X}$ are metrizable. It follows from the proof of Theorem \ref{s2t2} that $\mathbf{X}$ is also compact and second-countable. Hence $\mathbf{CUC}$ implies (ii) and (iii).\medskip

Now, to show that each of (i)--(iii) implies $\mathbf{CUC}$, we fix a disjoint family $\mathcal{A}=\{A_n: n\in\omega\}$ of non-empty countable sets and put $\mathbf{A}_n=\langle A_n, \mathcal{P}(A_n)\rangle$. 

($(i)\rightarrow\mathbf{CUC}$) Assume that (i) holds. Take an element $\infty\notin\bigcup\mathcal{A}$. For every $n\in\omega$,  put $H_n=A_n\cup\{\infty\}$ and $\mathbf{H}_n=\langle H_n, \mathcal{P}(H_n)\rangle$. By our assumption, the product $\mathbf{H}=\prod\limits_{n\in\omega}\mathbf{H}_n$ is separable. Let $D=\{a_i: i\in\omega\}$ be a countable dense set in $\mathbf{H}$. Then, for every $n\in\omega$, $H_n=\{a_i(n): i\in\omega\}$, which implies that $\bigcup\mathcal{A}$ is countable. Hence (i) implies $\mathbf{CUC}$.\medskip

($((ii)\vee(iii))\rightarrow\mathbf{CUC}$) Let $(\infty_n)_{n\in\omega}$ be a sequence of pairwise distinct elements such that none of $\infty_n$ is a member of $\bigcup\mathcal{A}$. If $A_n$ is finite, we put $\mathbf{Z_n}=\mathbf{A}_n$ and $Z_n=A_n$. If $A_n$ is infinite, let $Z_n=A_n\cup\{\infty_n\}$ and let $\mathbf{Z_n}$ be the one-point Hausdorff compactification of $\mathbf{A}_n$ such that $\infty_n$ is the unique accumulation point of $\mathbf{Z}_n$.  Let $\mathbf{Z}=\bigoplus\limits_{n\in\omega}\mathbf{Z}_n$. Assuming (ii), we can fix a countable base $\mathcal{B}$ of $\mathbf{Z}$. Then $\bigcup\limits_{n\in\omega}\{\{a\}: a\in A_n\}\subseteq\mathcal{B}$. This implies that $\bigcup\mathcal{A}$ is countable. Hence (ii) implies $\mathbf{CUC}$. To show that (iii) implies $\mathbf{CUC}$, without loss of generality, we may assume that $\prod\limits_{n\in\omega}A_n\neq\emptyset$. Assuming (iii), we deduce that the space $\mathbf{Z}^{\ast}=\prod\limits_{n\in\omega}\mathbf{Z}_n$ is second-countable. We fix a countable base $\mathcal{B}^{\ast}$ of $\mathbf{Z}^{\ast}$ and a point $c\in\prod\limits_{n\in\omega}A_n$. For every $n\in\omega$, we put $A_n(c)=\{x\in\prod_{i\in\omega}A_i: (\forall i\in\omega\setminus\{n\})x(i)=c(i)\}$,   $\mathcal{B}^{\ast}_n(c)=\{U\cap A_n(c): U\in\mathcal{B}^{\ast}\}$ and $\mathcal{B}^{\ast}(c)=\bigcup\limits_{n\in\omega}\mathcal{B}^{\ast}_n(c)$. Then the family $\mathcal{B}^{\ast}(c)$ is countable; thus, since $\bigcup\limits_{n\in\omega}\{\{z\}: z\in A_n(c)\}\subseteq\mathcal{B}^{\ast}(c)$, the set $\bigcup\limits_{n\in\omega}\{\{z\}: z\in A_n(c)\}$ is countable. This implies that $\bigcup\mathcal{A}$ is countable. Therefore, (iii) implies $\mathbf{CUC}$. 

\end{proof}

\begin{proposition}
\label{s2t19}$(\mathbf{ZF})$
\begin{enumerate}
\item[(i)] $\mathbf{CAC}_{\omega}$ is equivalent to the sentence: For every disjoint family $\{A_{i}:i\in 
\mathbb{N}\}$ of denumerable sets, there exists a denumerable set $H$ such
that, for every $i\in \mathbb{N}$, the set $A_{i}\cap H$ is infinite.

\item[(ii)] $\mathbf{CAC}(\leq 2^{\aleph_0})$ is equivalent to the sentence: For every family $\{\langle X_n, \tau_n\rangle: n\in\mathbb{N}\}$  of non-empty metrizable spaces such that, for each $n\in\mathbb{N}$, $X_n$ is equipotent to a subset of $\mathbb{R}$, the family $\{X_n: n\in\mathbb{N}\}$ has choice function.
\end{enumerate}
\end{proposition}
\begin{proof} (i) We fix a disjoint family $\mathcal{A}=\{A_i: i\in\mathbb{N}\}$ of denumerable sets. 

($\leftarrow$)If there exists a denumerable set $\{x_n: n\in\omega\}$ such that, for every $i\in\mathbb{N}$, $A_i\cap \{x_n: n\in\omega\}\neq\emptyset$, then, we can define a choice function $\psi$ of $\mathcal{A}$ by putting, for every $i\in\mathbb{N}$, $\psi(i)=x_{\min\{n\in\omega: x_n\in A_i\}}$. 

($\rightarrow$) For every pair $\langle i,n\rangle\in\mathbb{N}\times\mathbb{N}$, we define  $$B_{i,n} = \{f\in A_{i}^{n}: f \text{ is an injection}\}.$$  Then $\mathcal{B} = \{B_{i,n} : \langle i,n\rangle \in\mathbb{N}\times\mathbb{N}\}$ is a denumerable, disjoint family of denumerable sets. Assuming $\mathbf{CAC}_{\omega}$, we fix a choice function $h$ of $\mathcal{B}$. Then, for every $i\in\omega$, the set $H_i=\bigcup_{n\in\omega}h(\langle i, n\rangle)[n]$ is a denumerable subset of $A_i$, and the set $H=\bigcup\limits_{i\in\mathbb{N}}H_i$ is denumerable. This completes the proof of (i).

(ii) For the proof of (ii), it suffices to notice that, given a family $\{X_n: n\in\mathbb{N}\}$ of non-empty sets, we have the family $\{\langle X_n, \mathcal{P}(X_n)\rangle: n\in\mathbb{N}\}$ of metrizable spaces. 
\end{proof}

\section{Models for $\mathbf{CUC}\wedge\neg\mathbf{CAC}(\leq 2^{\aleph_0})$}
\label{s3}

\subsection{The aim of Section \ref{s3}}
In Section \ref{s3}, to establish several new facts about countable products of metrizable spaces and, in particular, to show that the implications of Theorem \ref{s2t17} are not reversible in $\mathbf{ZF}$, we are concerned with statements $\mathbf{\Phi}_1$--$\mathbf{\Phi}_3$ defined as follows:

$$\mathbf{\Phi_1}= \mathbf{CUC}\wedge\neg\mathbf{CSM}_{le}\wedge\neg\mathbf{IDI}\wedge\neg\mathbf{CAC}(\leq 2^{\aleph_0}),$$
$$\mathbf{\Phi_2}=\mathbf{CUC}\wedge\mathbf{IDI}\wedge\mathbf{WOAC}_{fin}\wedge\neg\mathbf{CSM}_{le}\wedge\neg\mathbf{CAC}(\leq 2^{\aleph_0}),$$
$$\mathbf{\Phi_3}=\mathbf{CUC}\wedge\mathbf{IDI}\wedge\neg\mathbf{CSM}_{le}\wedge\neg\mathbf{CAC}(\leq 2^{\aleph_0}).$$

We recall that $\mathbf{IDI}$, $\mathbf{WOAC}_{fin}$ and $\mathbf{CUC}$ are pairwise independent in both $\mathbf{ZFA}$ and $\mathbf{ZF}$ (see \cite{hr}). We also recall that the following result was established in \cite{chhkr0}:

\begin{theorem}
\label{s3t3.1}
(Cf. \cite[Theorem 10(iii) and its proof ]{chhkr0}.) The conjunction $\mathbf{CUC}\wedge\neg\mathbf{CSM}_{le}$ has a permutation model and is transferable to a model of $\mathbf{ZF}$. Furthermore, $\mathbf{CUC}$ does not imply $\mathbf{CAC}(\leq 2^{\aleph_0})$ in $\mathbf{ZFA}$.
\end{theorem}

We \emph{strengthen} the above result by establishing  that $\mathbf{\Phi}_3$ has a $\mathbf{ZF}$-model (see Theorem \ref{s3t9}), and $\mathbf{\Phi_2}$ has a permutation model (see Theorems \ref{s3t4} and \ref{s3t7}). We also construct in Section \ref{s3.3} a new permutation model for $\mathbf{\Phi}_2$. Moreover, we show that the reverse implications to, respectively, (i)--(iv) of Theorem \ref{s2t17} are all independent of $\mathbf{ZF}$ (see Theorem \ref{s3t10} and Corollary \ref{s3c11}).

\subsection{Notation, terminology and some useful tools}

\begin{notation}
\label{n1}
\begin{enumerate}
\item[(i)] For a linearly ordered set $\langle L, \leq\rangle$, we denote by $A(L,\leq)$ the group of all order-automorphisms of $\langle L, \leq\rangle$.
\item[(ii)] For a topological space $\mathbf{X}=\langle X, \tau\rangle$, we denote by $\Aut(\mathbf{X})$ (or by $\Aut(X, \tau)$ the group  of all autohomeomorphisms of $\mathbf{X}$ (i.e., homeomorphisms of $\mathbf{X}$ onto itself).
\end{enumerate}
\end{notation}

We shall apply the following useful theorems:

\begin{theorem}
\label{lem:dt}
(Droste and Truss \cite[p. 31]{dt}.) Let $\leq$ be the usual linear ordering of $\mathbb{R}$. Then the only subgroup of $A(\mathbb{R},\leq)$ having index $<|\mathbb{R}|$ is $A(\mathbb{R},\leq)$ itself.
\end{theorem}

In what follows, we denote by ${\mathbf{\omega}}^{\omega}$ the Baire space, that is, the product $\langle \omega, \mathcal{P}(\omega)\rangle^{\omega}$. Let $\tau_b$ denote the topology of $\mathbf{\omega}^{\omega}$. To prove that the new model described below satisfies $\mathbf{CUC}\wedge\mathbf{WOAC}_{fin}$, we shall apply the following result due to Trust:

\begin{theorem}
\label{thm:t}
(Truss \cite[Theorem 3.9]{tr}.) The only subgroup of $\Aut(\mathbb{N}^{\omega})$ having index $< |\mathbb{R}|$ is $\Aut(\mathbb{N}^{\omega})$ itself.
\end{theorem}

We recall that the Baire space $\mathbb{N}^{\omega}$ is homeomorphic to the subspace $\mathbb{P}$ of irrationals of the space $\mathbb{R}$ (equipped with the natural topology $\tau_{nat}$). Then the following theorem is equivalent to Theorem \ref{thm:t}.

\begin{theorem}
\label{thm1:t}
The only subgroup of $\Aut(\mathbb{P})$ having index $< |\mathbb{R}|$ is $\Aut(\mathbb{P})$ itself.
\end{theorem}

For a set $X$,  $\Sym(X)$ denotes the group of all permutations of $X$.

\begin{notation}
\label{n2}
In all the constructions of permutations models in Sections \ref{s3} and \ref{s4}, we assume that our ground model $\mathcal{M}=\mathcal{M}[G]$ is a fixed model of $\mathbf{ZFA+AC}$ whose set $A$ of atoms is expressed in $\mathcal{M}$ as the union of a disjoint denumerable family $\mathcal{A}=\{A_n: n\in\omega\}$ of infinite sets having some properties dependent on which permutation model we want to describe. Then $\mathcal{I}=\{ E\subseteq A: (\exists S\in [\omega]^{<\omega}): E\subseteq \bigcup_{n\in S}A_n\}$.  As soon as we choose a subgroup $\mathcal{G}$ of $\Sym(A)$ for a particular permutation model under construction,  $\mathcal{I}$ is a normal ideal of subsets of $A$, and $\mathcal{F}$ denotes the normal filter of subgroups of $\mathcal{G}$ generated by $\{\fix_{\mathcal{G}}(E): E\in\mathcal{I}\}$.
\end{notation}

\subsection{A permutation model for $\mathbf{\Phi_1}$}
\label{s3.1}

Let us denote by $\mathcal{N}(\cite{chhkr0})$ the permutation model from \cite[proof of Theorem 10(iii)]{chhkr0} in which $\mathbf{CUC}\wedge\neg\mathbf{CSM}_{le}$ is true. It was shown in \cite[proof of Theorem 10(iii)]{chhkr0} that there are amorphous sets in $\mathcal{N}(\cite{chhkr0})$. Therefore, $\mathbf{IDI}$ is false in $\mathcal{N}(\cite{chhkr0})$. It was also shown in \cite[proof of Theorem 10(iii)]{chhkr0} that $\mathbf{CAC}(\leq 2^{\aleph_0})$ is false in $\mathcal{N}(\cite{chhkr0})$. To summarize, let us write down the following theorem:

\begin{theorem}
\label{s3t2}
There exists a Fraenkel-Mostowski model in which $\mathbf{\Phi_1}$ is true. For instance, $\mathbf{\Phi_1}$ is true in $\mathcal{N}(\cite{chhkr0})$. 
\end{theorem}

\subsection{The known permutation model for $\mathbf{\Phi_2}$}
\label{s3.2}

Let us apply the permutation model constructed by K. Keremedis and E. Tachtsis in \cite[proof of Theorem 14]{kt}. Since we want to prove new facts about this model, let us describe it briefly but a little more precisely than in \cite{kt}.

In what follows in this section, we use Notation \ref{n2}.

\begin{definition}
\label{s3d8}
 We assume that, in the ground model $\mathcal{M}$ of Notation \ref{n2}, for every $n\in\omega$,  $A_n=\{a_{n, x}: x\in\mathbb{R}\}$ where the mapping $f_n:\mathbb{R}\to A_n$, defined by $f_n(x)=a_{n,x}$ for $x\in\mathbb{R}$, belongs to $\mathcal{M}$ and is a bijection. Let $\leq$ be the standard linear order of $\mathbb{R}$ in $\mathcal{M}$. In $\mathcal{M}$, for every $n\in\omega$, let $\leq_n$ be the linear order on $A_n$ defined as follows:  
$$(\forall x,y\in\mathbb{R}) (a_{n,x}\leq_n a_{n,y}\leftrightarrow x\leq y).$$
In $\mathcal{M}$, let $\mathcal{G}=\{\pi\in \Sym(A):  (\forall n\in\omega)(\pi\upharpoonright A_n\in A(A_n,\leq_n))\}$.  We denote by $\mathcal{N}_{\mathbb{R}}$ the permutation model determined by $\mathcal{M}$, $\mathcal{G}$ and the normal filter $\mathcal{F}$ (equivalently, the normal ideal $\mathcal{I}$).
\end{definition} 

The model $\mathcal{N}_{\mathbb{R}}$ is just the permutation model $\mathcal{N}$ defined in  \cite[proof of Theorem 14]{kt}. 

For every $n\in\omega$, we have $\leq_n\in\mathcal{N}_{\mathbb{R}}$ because $\sym_{\mathcal{G}}(\leq_{n})=\mathcal{G}\in\mathcal{F}$. Furthermore, for every $n\in\omega$, $f_n\in\mathcal{N}_{\mathbb{R}}$ because $A_n$ is a support of $f_n$ (see \cite[proof of Theorem 14]{kt}). In \cite[proof of Theorem 14]{kt}, it was observed (without proof) that $\{A_n: n\in\omega\}$ is (in $\mathcal{N}_{\mathbb{R}}$) a denumerable family of non-empty sets without a choice function in $\mathcal{N}_{\mathbb{R}}$. For the convenience of readers, we include here an argument that a deeper fact holds. Indeed, we have the following proposition. 

\begin{proposition}
\label{s3cl1}
The family $\mathcal{A}=\{A_n: n\in\omega\}$ is denumerable in $\mathcal{N}_{\mathbb{R}}$ but $\mathcal{A}$ does not have a partial multiple choice function in $\mathcal{N}_{\mathbb{R}}$. Hence: 
$$\mathcal{N}_{\mathbb{R}}\models\neg\mathbf{CAC}(\leq 2^{\aleph_0}).$$
\end{proposition}
\begin{proof}
It is obvious that $\mathcal{A}\in\mathcal{N}_{\mathbb{R}}$. Moreover, $\mathcal{A}$ is denumerable in $\mathcal{N}_{\mathbb{R}}$ since $\fix_{\mathcal{G}}(\mathcal{A})=\mathcal{G}\in\mathcal{F}$. We argue that $\mathcal{A}$ has no partial choice function in $\mathcal{N}_{\mathbb{R}}$ (and thus has no choice function in $\mathcal{N}_{\mathbb{R}}$). Assume the contrary. Let $\mathcal{B}$ be an infinite subfamily of $\mathcal{A}$ with a choice function $f\in\mathcal{N}_{\mathbb{R}}$. Let, for some finite $S\subsetneq\omega$, the set $E=\bigcup\{A_n:n\in S\}$ be a support of $f$. Since $\mathcal{B}$ is infinite, there exists $m\in\omega\setminus S$ such that $A_m\in\mathcal{B}$. Consider a $\phi\in \mathcal{G}$ such that $\phi\in\fix_{\mathcal{G}}(A\setminus A_m)$ and $\phi(f(A_{m}))\ne f(A_{m})$. Then $\phi\in\fix_{\mathcal{G}}(E)$, but $\phi(f)\ne f$ contradicting that $E$ is a support of $f$. Since, for every $n\in\omega$, $\leq_n\in\mathcal{N}_{\mathbb{R}}$, that $\mathcal{A}$ does not have a partial choice function in $\mathcal{N}_{\mathbb{R}}$ implies that $\mathcal{A}$ does not have a partial multiple choice function in $\mathcal{N}_{\mathbb{R}}$ because, for every $n\in\omega$, every non-empty finite subset of $A_n$ has a minimal element in the linearly ordered set $\langle A_n, \leq_n\rangle$. 
\end{proof}

It was shown in \cite[Theorem 14]{kt} that $\mathbf{CPM}_{le}$ is false in $\mathcal{N}_{\mathbb{R}}$. Therefore, by Theorem \ref{s1t11}, $\mathbf{CSM}_{le}$ is also false in $\mathcal{N}_{\mathbb{R}}$. Let us give a different from that in \cite{kt} argument for it, to show new facts about $\mathcal{N}_{\mathbb{R}}$. 

For every $n\in\omega$, let $\tau_{n}$ be the order topology on $A_n$ induced by the ordering $\leq_{n}$ on $A_{n}$. Since every $\phi\in \mathcal{G}$ fixes the base for $\tau_{n}$ comprising all open intervals $(a_{n,x},a_{n,y})$ ($x,y\in\mathbb{R}$) in the ordering $\leq_{n}$ on $A_n$, it follows that $\sym_{\mathcal{G}}(\tau_n)=\mathcal{G}$, and thus, for every $n\in\omega$,  $\tau_n\in\mathcal{N}_{\mathbb{R}}$. Furthermore,  for $n\in\omega$, the mapping $d_n:A_n\times A_{n}\rightarrow\mathbb{R}$ defined by $d_{n}(a_{n,x},a_{n,y})=|x-y|$ is a metric on $A_{n}$ which induces $\tau_{n}$, and $d_n\in\mathcal{N}_{\mathbb{R}}$ because $A_n$ is a support of each element of $d_n$. Hence, for every $n\in\omega$,  $\langle A_n,\tau_n\rangle$ is metrizable in $\mathcal{N}_{\mathbb{R}}$. It was shown in \cite[proof of Theorem 14]{kt} that  $\{\langle A_n, \tau_n\rangle: n\in\omega\}\in\mathcal{N}_{\mathbb{R}}$ and the family $\{\langle A_n, \tau_n\rangle: n\in\omega\}$ is denumerable in $\mathcal{N}_{\mathbb{R}}$. 

\begin{proposition} 
\label{s3cl2}
The direct sum $\mathbf{A}=\bigoplus_{n\in\omega}\langle A_n, \tau_n\rangle$ is not metrizable in $\mathcal{N}_{\mathbb{R}}$. In consequence, $\{d_n: n\in\omega\}\notin\mathcal{N}_{\mathbb{R}}$ and
$$\mathcal{N}_{\mathbb{R}}\models\neg\mathbf{CSM}_{le}.$$
\end{proposition} 
\begin{proof}
If $\mathbf{A}$ were metrizable in $\mathcal{N}_{\mathbb{R}}$, then, due to $\mathcal{G}$ and $\mathcal{F}$, all but finitely many of the spaces $\langle A_n, \tau_n\rangle$ would have to be discrete, which is absurd. If the family $\{d_n: n\in\omega\}$ were in $\mathcal{N}_{\mathbb{R}}$, then $\mathbf{A}$ would be metrizable in $\mathcal{N}_{\mathbb{R}}$. We leave the details of the above observations as an easy exercise for the interested readers. The space $\mathbf{A}$ witnesses that $\mathbf{CSM}_{le}$ fails in $\mathcal{N}_{\mathbb{R}}$.
\end{proof}

\begin{proposition}
\label{s3cl3}
$\mathcal{N}_{\mathbb{R}}\models\mathbf{IDI}\wedge\mathbf{CUC}$.
\end{proposition}
\begin{proof}
This was shown in \cite[proof of Theorem 14]{kt} by applying Theorem \ref{lem:dt}.
\end{proof}

The forthcoming two theorems about $\mathcal{N}_{\mathbb{R}}$ are new.

\begin{theorem}
\label{s3cl4}
$\mathcal{N}_{\mathbb{R}}\models\mathbf{WOAC}_{fin}$.
\end{theorem}
\begin{proof}
Let $\mathcal{U}=\{X_{\alpha}:\alpha\in\kappa\}$, where $\kappa$ is an infinite well-ordered cardinal, be an infinite, well-ordered family of non-empty finite sets in $\mathcal{N}_{\mathbb{R}}$ (we assume that the mapping $\kappa\ni\alpha\mapsto X_{\alpha}\in\mathcal{U}$ is a bijection in $\mathcal{N}_{\mathbb{R}}$). Let $S\subsetneq\omega$ be a finite set such that, for every $\alpha\in\kappa$, the set $E=\bigcup\{A_{n}:n\in S\}$ is a support of $X_{\alpha}$. We assert that $E$ is a support of every element of the set $X=\bigcup\mathcal{U}$. This will yield that $X$ is well-orderable in $\mathcal{N}_{\mathbb{R}}$. Towards a  contradiction, we assume that there exist $\alpha\in\kappa$ and $x\in X_{\alpha}$ such that $E$ is not a support of $x$.

Let $S^{\prime}$ be a finite subset of $\omega$ such that $S\cap S^{\prime}=\emptyset$ and, for $E^{\prime}=\bigcup\{A_{n}:n\in S'\}$, the set $E\cup E^{\prime}$ is a support of $x$.
We assert that there exist $s\in S^{\prime}$ and $\eta\in \fix_{\mathcal{G}}(A\setminus A_{s})$ such that $\eta(x)\ne x$. Indeed, since $E$ is not a support of $x$, there exists $\psi\in\fix_{\mathcal{G}}(E)$ such that $\psi(x)\ne x$. Let $\phi$ be the permutation of $A$ which agrees with $\psi$ on $E^{\prime}$ and is the identity outside of $E^{\prime}$; hence, $\phi\in\fix_{\mathcal{G}}(E)$ since $E\cap E^{\prime}=\emptyset$ and $\phi\in\fix_{\mathcal{G}}(A\setminus E^{\prime})$. For every $s\in S^{\prime}$, we denote by $\phi_s$ the permutation of $A$ such that $\phi_s$ agrees with $\psi$ on $A_s$ and is the identity outside of $A_s$. Then $\phi$ is the composition of the permutations $\phi_s$ with $s\in S^{\prime}$ and, for each $s\in S^{\prime}$, $\phi_{s}\in\fix_{\mathcal{G}}(A\setminus A_{s})$. Since $\psi$ and $\phi$ agree on the support $E\cup E^{\prime}$ of $x$, we have $\phi(x)=\psi(x)$, and thus $\phi(x)\ne x$ (for $\psi(x)\ne x$). This yields that there exists $s\in S'$ such that $\phi_{s}(x)\ne x$. Letting $\eta=\phi_{s}$, we obtain that $\eta\in \fix_{\mathcal{G}}(A\setminus A_{s})$ and $\eta(x)\ne x$ as asserted.

Let $$Y = \{ \sigma(x) : \sigma \in \fix_{\mathcal{G}}(A\setminus A_{s}) \}.$$
Since $x\in X_{\alpha}$, $\fix_{\mathcal{G}}(A\setminus A_{s})\subseteq\fix_{\mathcal{G}}(E)$ and $E$ is a support of $X_{\alpha}$, we infer that $Y\subseteq X_{\alpha}$. Hence $Y$ is finite. Let
$$H=\{\sigma\in\fix_{\mathcal{G}}(A\setminus A_{s}):\sigma(x)=x\}.$$ 
$H$ is a proper subgroup of $\fix_{\mathcal{G}}(A\setminus A_{s})$ since $\eta\in\fix_{\mathcal{G}}(A\setminus A_{s})\setminus H$. Furthermore, as $Y$ is finite, so is the quotient group $\fix_{\mathcal{G}}(A\setminus A_{s})/H$, and thus $|\fix_{\mathcal{G}}(A\setminus A_{s}):H|<|\mathbb{R}|$. This, together with Theorem \ref{lem:dt} and the fact that $\fix_{\mathcal{G}}(A\setminus A_{s})$ is isomorphic to $A(A_{s},\leq_s)$ (and thus to $A(\mathbb{R},\leq)$), implies that $H=\fix_{\mathcal{G}}(A\setminus A_{s})$, which is a contradiction.

By the above arguments, we conclude that $\mathbf{WOAC}_{fin}$ is true in $\mathcal{N}_{\mathbb{R}}$.  
\end{proof}

\begin{theorem}
\label{s3t4}
$\mathcal{N}_{\mathbb{R}}\models \mathbf{\Phi}_2$. 
\end{theorem}
\begin{proof}
It suffices to apply Propositions \ref{s3cl1}--\ref{s3cl3} and Theorem \ref{s3cl4}.
\end{proof}

\begin{remark}
\label{s3r2}
(i) Note that the statement ``The union of a well-orderable family of well-orderable sets is well-orderable'' (Form 231 in \cite{hr}) is false in $\mathcal{N}_{\mathbb{R}}$. Indeed, for each $n\in\omega$, $A_n$ is well-orderable in $\mathcal{N}_{\mathbb{R}}$ (since $\fix_{\mathcal{G}}(A_n)\in\mathcal{F}$), but their union, $A$, is not.

(ii) Similarly to the proof of Proposition \ref{s3cl1}, one may verify that for the denumerable family $\mathcal{A}$ there is no function $f\in\mathcal{N}_{\mathbb{R}}$ whose domain is some infinite $\mathcal{B}\subseteq\mathcal{A}$ such that, for all $B\in\mathcal{B}$, $f(B)$ is a non-empty countable subset of $B$. 

(iii)  We recall that, in every Fraenkel--Mostowski model, $\mathbf{AC}_{fin}$ implies $\mathbf{AC}_{WO}$ (see \cite{how}). Since, by (i) and Proposition \ref{s3cl1}, $\mathbf{AC}_{WO}$ is false in $\mathcal{N}_{\mathbb{R}}$, $\mathbf{AC}_{fin}$ is also false in $\mathcal{N}_{\mathbb{R}}$.
\end{remark}

\subsection{A new permutation model for $\mathbf{\Phi_2}$}
\label{s3.3}

We consider it important to provide the readers with further insight and new information on models of set theory lacking $\mathbf{CAC}(\leq 2^{\aleph_0})$ and also failing to satisfy $\mathbf{CSM}_{le}$. Thus, we define here a \emph{new permutation model} in which $\mathbf{\Phi}_2$ is true. As in Section \ref{s3.2}, in what follows, we use Notation \ref{n2}.

\begin{definition} We assume that, in the ground model $\mathcal{M}$ of Notation \ref{n2}, for each $n\in\omega$,  $A_{n}=\{a_{n,r}:r\in \mathbb{P}\}$ where the mapping $h_n:\mathbb{P}\to A_n$, defined by $h_n: \mathbb{P}\ni r\mapsto a_{n,r}$, is a bijection, and  $A_{n}$ is equipped with the topology $\tau_{n}$ so that $h_n$ is a homeomorphism of  $\mathbb{P}$ onto $\langle A_{n},\tau_{n}\rangle$. Let $\mathcal{G}=\{\phi\in\Sym(A): (\forall n\in\omega)(\phi\upharpoonright A_{n}\in \Aut(A_n, \tau_n))\}$. Then $\mathcal{N}_{\mathbb{P}}$ is the permutation model determined by $\mathcal{M}$, $\mathcal{G}$ and $\mathcal{I}$.
\end{definition}

We notice that, for every $n\in\omega$, $\tau_{n}\in\mathcal{N}_{\mathbb{P}}$ because $\sym_{\mathcal{G}}(\tau_{n})=\mathcal{G}\in\mathcal{F}$. Of course, for every $n\in\omega$, $|A_n|=|\mathbb{R}|$ in $\mathcal{N}_{\mathbb{P}}$.
That, for each $n\in\omega$,  $\langle A_{n},\tau_{n}\rangle$ is metrizable in $\mathcal{N}_{\mathbb{P}}$ follows from the fact that the mapping $d_{n}:A_n\times A_{n}\rightarrow\mathbb{R}$ defined by $d_n(a_{n,r},a_{n,t})=|r-t|$ is a metric on $A_{n}$ which induces $\tau_{n}$, and $d_{n}\in\mathcal{N}_{\mathbb{P}}$ because $A_{n}$ is a support of every element of $d_{n}$. 

\begin{theorem}
\label{s3t7}
$\mathcal{N}_{\mathbb{P}}\models \mathbf{\Phi}_2$.
\end{theorem}

\begin{proof}
In much the same way, as for the model $\mathcal{N}_{\mathbb{R}}$ in Section \ref{s3.2}, one can verify that the family $\mathcal{A}$ is denumerable in $\mathcal{N}_{\mathbb{P}}$ but $\mathcal{A}$ does not have a partial multiple choice function in $\mathcal{N}_{\mathbb{P}}$; moreover, the family $\{\langle A_n,\tau_n\rangle: n\in\omega\}$ belongs to $\mathcal{N}_{\mathbb{P}}$ and is denumerable in $\mathcal{N}_{\mathbb{P}}$ and, for every $n\in\omega$, the space $\langle A_n, \tau_n\rangle$ is metrizable in $\mathcal{N}_{\mathbb{P}}$  but the space $\mathbf{A}=\bigoplus_{n\in\omega}\langle A_n, \tau_n\rangle$ is not metrizable in $\mathcal{N}_{\mathbb{P}}$. Hence both $\mathbf{CAC}(\leq 2^{\aleph_0})$ and $\mathbf{CSM}_{le}$ are false in $\mathcal{N}_{\mathbb{P}}$. That $\mathbf{CUC}$ and $\mathbf{IDI}$ are both true in $\mathcal{N}_{\mathbb{P}}$ can be shown by mimicking the proof of Theorem 14 in \cite{kt} and applying Theorem \ref{thm1:t}. That $\mathbf{WOAC}_{fin}$ is true in $\mathcal{N}_{\mathbb{P}}$ can be proved similarly to the proof of Theorem \ref{s3cl4}, using this time Theorem \ref{thm1:t} in place of Theorem \ref{lem:dt}. 
\end{proof}

\subsection{Independence results}
\label{s3.4}

To establish our independence results, we need the following simple lemma:

\begin{lemma}
\label{s317}
The following statements are all boundable, so also injectively boundable (up to equivalence): $\neg\mathbf{IDI}$,  $\neg\mathbf{CSM}_{le}$, $\neg\mathbf{CPM}_{le}(CS, M)$ and $\neg\mathbf{CPM}_{le}(CS, S)$. The statements $\mathbf{IDI}$ and $\mathbf{CUC}$ are both injectively boundable.
\end{lemma}
\begin{proof}
It is known, for instance, from \cite[Note 3, p. 285]{hr} that $\mathbf{IDI}$ and $\mathbf{CUC}$ are both injectively boundable. It was shown in \cite[Proposition 9]{ktw1} that $\mathbf{UT}(\aleph_0, \aleph_0, cuf)$ is injectively boundable. One can easily check the four negations in the first statement of the lemma are boundable; thus, by Fact \ref{f:1}, they are equivalent to injectively boundable statements.
\end{proof}

\begin{theorem}
\label{s3t9}
The statements $\mathbf{\Phi}_1$ and $\mathbf{\Phi_3}$ are transferable. In consequence, $\mathbf{IDI}$ is independent of $\mathbf{ZF+\Phi_3}$. 
\end{theorem}
\begin{proof}
By Lemma \ref{s317}, both $\mathbf{\Phi}_1$ and $\mathbf{\Phi}_3$ are equivalent to conjunctions of injectively boundable statements. Moreover, in view of Theorems \ref{s3t2} and \ref{s3t4} (or \ref{s3t7}), both $\mathbf{\Phi}_1$ and $\mathbf{\Phi}_3$ have permutation models.  Therefore, by Theorem \ref{thm:Pinc}, both $\mathbf{\Phi}_1$ and $\mathbf{\Phi}_3$ are transferable statements. This, together with the definitions of $\mathbf{\Phi}_1$ and $\mathbf{\Phi}_3$, implies that $\mathbf{IDI}$ is independent of $\mathbf{ZF+\Phi_3}$. 
\end{proof}

\begin{theorem}
\label{s3t10}
None of the implications of Theorem \ref{s2t17} are reversible in $\mathbf{ZF}$.
\end{theorem}
\begin{proof}
Let us consider Truss' Model I denoted by $\mathcal{M}12(\aleph)$ in \cite{hr} (see also \cite{tr0}). It is known that it is true in $\mathcal{M}12(\aleph)$ that, for every well-ordered family $\mathcal{W}$ of well-orderable sets, the union $\bigcup\mathcal{W}$ is well-orderable (see \cite[Form 231 and p. 155]{hr}). This implies that  $\mathbf{CAC}_{\omega}$ is true in $\mathcal{M}12(\aleph)$; thus, $\mathbf{CMC}_{\omega}$ is also true in $\mathcal{M}12(\aleph)$. It is known that $\aleph_{1}$ (that is, the well-ordered cardinal $\omega_1$) is singular in $\mathcal{M}12(\aleph)$, which implies that $\mathbf{UT}(\aleph_0, \aleph_0, cuf)$ is false in $\mathcal{M}12(\aleph)$. Hence, the second implication of Theorem \ref{s2t17}(ii) is false in $\mathcal{M}12(\aleph)$. This means that $\mathbf{CMC}_{\omega}$ does not imply $\mathbf{UT}(\aleph_0, \aleph_0, cuf)$ in $\mathbf{ZF}$.

To prove that none of the other implications of Theorem \ref{s2t17} are reversible in $\mathbf{ZF}$, let us consider the permutation model  $\mathcal{N}_{\mathbb{R}}$ described in Definition \ref{s3d8}. We use the same notation, as in Definition \ref{s3d8}.  

Now, let us modify and clarify an idea that appeared in \cite[proof of Theorem 14 (ii)]{kt}. For every $n\in\omega$, let $X_n=A_{n}\cup\{2n, 2n+1\}$ and also let $\prec_n$ be the binary relation on $X_{n}$ defined by: $2n\prec_n x\prec_n 2n+1$ for all $x\in A_{n}$ and, for $x,y\in A_{n}$, $x\prec_n y$ if and only if $x<_{n} y$. Then $\preceq_{n}=\prec_{n}\cup\{\langle x,x\rangle:x\in X_{n}\}$ is a linear ordering on $X_n$ since $\leq_n$ is a linear ordering on $A_n$. Moreover, $\preceq_{n}\in\mathcal{N}_{\mathbb{R}}$ because $\sym_{\mathcal{G}}(\preceq_{n})=\mathcal{G}\in\mathcal{F}$.

For every $n\in\omega$, let $\mathcal{T}_{n}$ be the order topology on $X_n$ induced by the ordering $\preceq_{n}$ on $X_{n}$. Then, for every $n\in\omega$,  $\mathcal{T}_{n}\in\mathcal{N}_{\mathbb{R}}$ because $\sym_{\mathcal{G}}(\mathcal{T}_{n})=\mathcal{G}\in\mathcal{F}$. Furthermore, as $\fix_{\mathcal{G}}(A_{n})\subseteq\fix_{\mathcal{G}}(\langle X_{n},\mathcal{T}_{n}\rangle)$, it follows that, for every $n\in\omega$, $\mathbf{X}_{n}=\langle X_n,\mathcal{T}_{n}\rangle$ is, in $\mathcal{N}_{\mathbb{R}}$, homeomorphic to the closed interval $[2n,2n+1]$ as a subspace of $\mathbb{R}$. Hence, for every $n\in\omega$,  $\mathbf{X}_{n}$ is a compact, separable, metrizable space in $\mathcal{N}_{\mathbb{R}}$.

If the product space $\mathbf{X}=\prod\limits_{n\in\omega}\mathbf{X}_{n}$ were metrizable in $\mathcal{N}_{\mathbb{R}}$, then, by \cite[Theorem 7]{chhkr0}, the direct sum $\bigoplus\limits_{n\in\omega}\mathbf{X}_n$ would be metrizable in $\mathcal{N}_{\mathbb{R}}$. But then, the direct sum $\bigoplus\limits_{n\in\omega}\langle A_n,\tau_n\rangle$ would be metrizable in $\mathcal{N}_{\mathbb{R}}$, contradicting Proposition \ref{s3cl2}. Thus $\mathbf{X}$ is not metrizable in $\mathcal{N}_{\mathbb{R}}$. Hence $\mathbf{CPM}_{le}(CS,M)$ is false in $\mathcal{N}_{\mathbb{R}}$. 

In the light of Remark \ref{s3r2}(ii), there is no function $f\in\mathcal{N}_{\mathbb{R}}$ whose domain is some infinite $\mathcal{B}\subseteq\mathcal{A}$ such that, for every $B\in\mathcal{B}$, $f(B)$ is a non-empty countable subset of $B$. This implies that $\mathbf{X}$ is not separable in $\mathcal{N}_{\mathbb{R}}$ and, in consequence, $\mathbf{CPM}_{le}(CS,S)$ is false in $\mathcal{N}_{\mathbb{R}}$. 

Finally, in view of Lemma \ref{s317}, the statement: 
$$\Phi=\mathbf{CUC}\wedge\neg\mathbf{CPM}_{le}(CS,M)\wedge\neg\mathbf{CPM}_{le}(CS,S)$$
is equivalent to a conjunction of injectively boundable statements. Since $\mathbf{\Phi}$ has a permutation model, it follows from Theorem \ref{thm:Pinc} that $\mathbf{\Phi}$ has a $\mathbf{ZF}$-model. Since $\mathbf{CUC}$ implies each of $\mathbf{CUC}_{\omega}$, $\mathbf{UT}(\aleph_0, \aleph_0, cuf)$ and $\mathbf{CAC}_{fin}$ in $\mathbf{ZF}$, it follows that, in every $\mathbf{ZF}$-model for $\mathbf{\Phi}$, the reverse implications to (i)--(iv) of Theorem \ref{s2t17}, respectively, are all false.  This completes the proof of the theorem.  
\end{proof}

\begin{corollary}
\label{s3c11}
The reverse implications to, respectively, (i)--(iv) of Theorem \ref{s2t17} are all independent of $\mathbf{ZF}$.
\end{corollary}
\begin{proof}
It suffices to apply Theorem \ref{s3t10} and notice that the reverse implications to, respectively, (i)--(iv) of Theorem \ref{s2t17} are all true in every model of $\mathbf{ZFC}$.
\end{proof}

\section{A new permutation model for the conjunction of $\mathbf{M}(C(\not\leq\aleph_{0}),\geq 2^{\aleph_0})\wedge\neg\mathbf{CMC}_{\omega}$ with some forms of type $\mathbf{CPM}_{le}(\square, \square)$}
\label{s4}
\subsection{Motivation}
\label{s4.1}
In Section \ref{s4}, we pay special attention to the new form $\mathbf{M}(C(\not\leq\aleph_{0}), \geq 2^{\aleph})$ (see Definition \ref{s1d7}) which follows from $\mathbf{M}(C, S)$ by the following simple proposition:

\begin{proposition}
\label{s4p1}
$(\mathbf{ZF})$ $\mathbf{M}(C,S)\rightarrow\mathbf{M}(C(\not\leq\aleph_{0}),\geq 2^{\aleph_{0}})\rightarrow\mathbf{CAC}_{fin}$.
\end{proposition}
\begin{proof}
The first implication is straightforward since every uncountable, separable, compact metrizable space is equipotent to $\mathbb{R}$ (cf. \cite{kk2}). 
 
For the second implication, assume $\mathbf{M}(C(\not\leq\aleph_{0}),\geq 2^{\aleph_{0}})$. Let $\mathcal{A}=\{A_{n}:n\in\omega\}$ be a denumerable disjoint family of non-empty finite sets. By way of contradiction, assume $\mathcal{A}$ has no choice function.

Let $A=\bigcup\mathcal{A}$ and $\mathbf{A}=\langle A,\mathcal{P}(A)\rangle$. Consider the one-point Hausdorff compactification $\mathbf{A}(\infty)$ of $\mathbf{A}$. Then $\mathbf{A}(\infty)$ has a cuf base, so it is metrizable by Theorem \ref{s1t15}(ii).  Suppose that $\mathcal{A}$ has no choice function. Then $A$ is uncountable. Thus, by $\mathbf{M}(C(\not\leq\aleph_{0}),\geq 2^{\aleph_{0}})$,  $|\mathbb{R}|\leq |A|$. This implies that $\mathbb{R}$ is a cuf set. But this is impossible since every cuf subset of $\mathbb{R}$ is countable. Hence, $\mathcal{A}$ has a choice function.   
\end{proof}

It is known from \cite{ktw1} that the first implication of Proposition \ref{s4p1} is not reversible in $\mathbf{ZF}$.  We do not know a satisfactory answer to the following questions:
\begin{question}
\label{s4q1}
\begin{enumerate}
\item[(i)] Are $\mathbf{M}(C(\not\leq\aleph_{0}),\geq 2^{\aleph_0})$ and $\mathbf{M}(C, S)$  equivalent in $\mathbf{ZF}$?
\item[(ii)] Does $\mathbf{CAC}_{fin}$ imply $\mathbf{M}(C(\not\leq\aleph_{0}),\geq 2^{\aleph_0})$ in $\mathbf{ZF}$?
\end{enumerate}
\end{question} 

In \cite[Section 6, Problem 2]{ktw1}, it is asked whether or not $\mathbf{M}(C,S)$ implies $\mathbf{CUC}$ in $\mathbf{ZF}$. It is also an open problem whether $\mathbf{M}(C,S)$ implies the weaker principle $\mathbf{CAC}_{\omega}$. Although we are still unable to solve these open problems, we can shed more light on them by showing that the conjunction $\mathbf{M}(C(\not\leq\aleph_{0}),\geq 2^{\aleph_0})\wedge\neg\mathbf{CAC}_{\omega}$ has a permutation model.

Let us define statements $\mathbf{\Phi}_4$ and $\mathbf{\Phi}_5$ as follows:
$$\mathbf{\Phi}_4=  \mathbf{IDI}\wedge\mathbf{WOAC}_{fin}\wedge\mathbf{M}(C(\not\leq\aleph_{0}),\geq 2^{\aleph_{0}}),$$
$$\mathbf{\Phi}_5=\neg\mathbf{CAC}(C, \mathbf{M}_{le})\wedge\neg\mathbf{CPM}_{le}(C,C)\wedge\neg\mathbf{CMC}_{\omega}.$$ 

In the subsequent Section \ref{s4.3}, we construct a \emph{new permutation model} $\mathcal{N}_{\mathbf{2}^{\omega}}$ in which the conjunction $\mathbf{\Phi}_4\wedge\mathbf{\Phi}_5$ is true.  Since, by Proposition \ref{s4p1}, $\mathbf{M}(C(\not\leq\aleph_{0}),\geq 2^{\aleph_{0}})$ is formally weaker than $\mathbf{M}(C,S)$, the result that there is a permutation model for $\mathbf{\Phi}_4\wedge\mathbf{\Phi}_5$ is a substantial action towards the resolution of \cite[Problem 2]{ktw1} and can be considered as a \emph{partial answer} to the question posed in \cite[Problem 2]{ktw1}. 

\begin{remark}
\label{s4r3} 
Let $\mathcal{N}$ be a permutation model in which $\mathbf{CMC}_{\omega}$ is false. Then $\mathbf{CUC}$ is false in $\mathcal{N}$; thus,  it follows from Theorem \ref{s2t17} that $\mathbf{CPM}_{le}(C,S)$ and $\mathbf{CPM}_{le}(CS, MS)$ are also false in $\mathcal{N}$. In particular, $\mathbf{CPM}_{le}(CS,2)$, being equivalent to $\mathbf{CPM}_{le}(CS, MS)$ by Theorem \ref{s2t2}(viii), is false in $\mathcal{N}$. Hence, for our promised new model $\mathcal{N}_{\mathbf{2}^{\omega}}$  we will have:
$$\mathcal{N}_{\mathbf{2}^{\omega}}\models (\mathbf{\Phi}_4\wedge\mathbf{\Phi}_5\wedge\neg\mathbf{CPM}_{le}(C,S)\wedge\neg\mathbf{CPM}_{le}(CS,2)).$$
\end{remark} 
 
\subsection{Properties of Cantor spaces as tools}
\label{s4.2}

In Section \ref{s4.3}, we will apply Cantor spaces to our new permutation model $\mathcal{N}_{\mathbf{2}^{\omega}}$.  Let us recall that a \emph{Cantor space} is a non-empty, compact, second-countable, dense-in-itself, zero-dimensional Hausdorff space.  An important $\mathbf{ZF}$-example of a Cantor space is the Cantor cube $\mathbf{2}^{\omega}$.  The following theorem is part of the folklore:

\begin{theorem}
\label{thm:prel1}
$(\mathbf{ZF})$ Any two Cantor spaces are homeomorphic. In particular, any two non-empty, clopen subsets of a Cantor space are homeomorphic.
\end{theorem}

By the compactness of $\mathbf{2}^{\omega}$, every clopen subset of $\mathbf{2}^{\omega}$ is a finite union of basic clopen sets of the form $[p]=\{f\in 2^{\omega}:p\subset f\}$, where $p\in 2^{F}$ for some finite $F\subsetneq\omega$. Thus, the collection of all clopen subsets of $\mathbf{2}^{\omega}$ is countable. Furthermore, $\Aut(\mathbf{2}^{\omega})$ is isomorphic to the automorphism group of the countable atomless Boolean algebra of all clopen subsets of $\mathbf{2}^{\omega}$. 

That $H$ is a subgroup of a group $G$ is denoted by $H\leq G$.

The following two theorems are of special importance for our reasoning in Section \ref{s4.3}:
\begin{theorem}
\label{thm:prel2}
(Anderson, \cite[Corollary 2]{A}.) The group $\Aut(\mathbf{2}^{\omega})$ is simple.
\end{theorem}
\begin{theorem}
\label{thm:prel3}
(Truss, \cite[Theorem 3.7 and its proof, Corollary 3.8]{tr}.) If $H$ is a subgroup of $\Aut(\mathbf{2}^{\omega})$ of index $<|\mathbb{R}|$, then for some finite partition $F$ of ${2}^{\omega}$ into clopen sets of $\mathbf{2}^{\omega}$, $\fix_{\Aut(\mathbf{2}^{\omega})}(F){\leq H}{\leq\fix_{\Aut(\mathbf{2}^{\omega})}(\{F\})}$.
\end{theorem}

Let us note that, for $F\subseteq\mathcal{P}(2^{\omega})$, every $\phi\in\fix_{\Aut(\mathbf{2}^{\omega})}(F)=\{\psi\in\Aut(2^{\omega}):(\forall C\in F)(\psi(C)=C)\}$ fixes every element of $F$ setwise, and every $\phi\in\fix_{\Aut(\mathbf{2}^{\omega})}(\{F\})$ fixes $F$ setwise.

 \subsection{The permutation model for $\mathbf{\Phi}_4\wedge\mathbf{\Phi}_5$}
 \label{s4.3}
 
To define our permutation model $\mathbf{\Phi}_4\wedge\mathbf{\Phi}_5$, in what follows, according to Notation \ref{n2}, we assume that $\mathcal{M}$ is a fixed model of $\mathbf{ZFA+AC}$, $\mathcal{A}=\{A_n: n\in\omega\}$ is a denumerable disjoint family in $\mathcal{M}$ such that $A=\bigcup
\mathcal{A}$ is the set of atoms of $\mathcal{M}$, and $\mathcal{I}=\{E\subseteq A: (\exists S\in [\omega]^{<\omega})E\subseteq\bigcup_{n\in S}A_n\}$. Furthermore, we assume here that $\mathcal{M}$ is chosen so that it holds in $\mathcal{M}$ that, for every $n\in\omega$, $A_n=\{a_{n,f}: f\in 2^{\omega}\}$, $\tau_n$ is a topology on $A_n$ such that the mapping defined by $2^{\omega}\ni f\mapsto a_{n,f}$  is a homeomorphism of $\mathbf{2}^{\omega}$ onto $\mathbf{A}_n=\langle A_n, \tau_n\rangle$.  Then, for every $n\in\omega$, $\Aut(\mathbf{A}_n)$ is isomorphic to $\Aut(\mathbf{2}^{\omega})$ in $\mathcal{M}$.
Let $\mathcal{G}=\{\phi\in\Sym(A): (\forall n\in\omega)(\phi\upharpoonright A_{n}\in \Aut(\mathbf{A}_n))\}$. We denote by $\mathcal{F}$ the filter of subgroups of $\mathcal{G}$ generated by $\{\fix_{\mathcal{G}}(E): E\in\mathcal{I}\}$. 

\begin{definition}
$\mathcal{N}_{\mathbf{2}^{\omega}}$ is the permutation model determined by $\mathcal{M}$, $\mathcal{G}$ and the normal ideal $\mathcal{I}$ (equivalently, the normal filter $\mathcal{F}$).
\end{definition}

We are going to demonstrate a detailed proof that $\mathbf{\Phi}_4$ and $\mathbf{\Phi}_5$ are both true in $\mathcal{N}_{\mathbf{2}^{\omega}}$.
Since we also want our prospective readers obtain illuminating information and further ideas on permutation models and their techniques, we give two proofs for the validity of $\mathbf{IDI}$ in $\mathcal{N}_{\mathbf{2}^{\omega}}$; the first of these proofs uses Theorem \ref{thm:prel2}, while the second one uses Theorem \ref{thm:prel3} and appears useful for the proof of $\mathbf{M}(C(\not\leq\aleph_{0}),\geq 2^{\aleph_{0}})$ in $\mathcal{N}_{\mathbf{2}^{\omega}}$.
Let us point out that Truss' proof of his Theorem 3.7 in \cite{tr} (see Theorem \ref{thm:prel3}) uses Anderson's result (see Theorem \ref{thm:prel2}), and that our first proof (via Anderson's theorem) of $\mathbf{IDI}$ in $\mathcal{N}_{\mathbf{2}^{\omega}}$, leading also to an easier proof of $\mathbf{WOAC}_{fin}$ in $\mathcal{N}_{\mathbf{2}^{\omega}}$, is more direct and much simpler than the second one; hence, it is interesting in its own right.  

We would also like to note that the arguments of the proofs of $\mathbf{IDI}$ in $\mathcal{N}_{\mathbf{2}^{\omega}}$ have substantial differences from the corresponding ones for the models $\mathcal{N}_{\mathbb{R}}$ and $\mathcal{N}_{\mathbb{P}}$ (see Sections \ref{s3.2} and \ref{s3.3}) whose group-theoretic keypoints cannot be applied to $\mathcal{N}_{\mathbf{2}^{\omega}}$.

\begin{theorem}
\label{t:main}
$\mathcal{N}_{\mathbf{2}^{\omega}}\models(\mathbf{\Phi}_4\wedge\mathbf{\Phi}_5)$.
\end{theorem}
 
\begin{proof}

For each $n\in\omega$, $\tau_{n}\in\mathcal{N}_{\mathbf{2}^{\omega}}$ since $\sym_{\mathcal{G}}(\tau_{n})=\mathcal{G}\in\mathcal{F}$. Furthermore, $\langle A_{n},\tau_{n}\rangle$ is compact and metrizable in $\mathcal{N}_{\mathbf{2}^{\omega}}$ since it is homeomorphic (in $\mathcal{N}_{\mathbf{2}^{\omega}}$) to the compact, metrizable space $\mathbf{2}^{\omega}$. Arguing similarly to Sections \ref{s3.2} and \ref{s3.3}, one can show that the family $\mathcal{A}$ is  denumerable in $\mathcal{N}_{\mathbf{2}^{\omega}}$ but $\mathcal{A}$  has no partial multiple choice function in $\mathcal{N}_{\mathbf{2}^{\omega}}$. Thus, $\mathbf{CAC}(C,\mathbf{M}_{le})$ is false in $\mathcal{N}_{\mathbf{2}^{\omega}}$. 

For each $n\in\omega$, let $\mathbf{X}_{n}=\mathbf{A}_{n}\bigoplus\{n\}$; note that $\omega\cap A=\emptyset$ since $\omega$ is in the kernel of $\mathcal{N}_{\mathbf{2}^{\omega}}$, i.e., $\omega$ is a pure set in $\mathcal{N}_{\mathbf{2}^{\omega}}$. As $\mathcal{A}$ has no choice function in $\mathcal{N}_{\mathbf{2}^{\omega}}$, the product space $\prod_{n\in\omega}\mathbf{X}_{n}$ is neither separable nor compact in $\mathcal{N}_{\mathbf{2}^{\omega}}$. Hence, $\mathbf{CPM}_{le}(C,S)$ and $\mathbf{CPM}_{le}(C,C)$ are also false in $\mathcal{N}_{\mathbf{2}^{\omega}}$.
  
\begin{claim}
\label{c:IDI_in_M}
$\mathcal{N}_{\mathbf{2}^{\omega}}\models\mathbf{IDI}$.
\end{claim}
\begin{proof}
\textbf{A.} Let $X\in \mathcal{N}_{\mathbf{2}^{\omega}}$, be an infinite, non-well-orderable set in $\mathcal{N}_{\mathbf{2}^{\omega}}$. Let $E=\bigcup\{A_{n}:n\in S\}$, where $S\subsetneq\omega$ is finite, be a support of $X$. Then there exists $x\in X$ such that $E$ is not a support of $x$. Let $E\cup E'$ be a support of $x$, where $E'=\bigcup\{A_{n}:n\in S'\}$ for some finite $S'\subsetneq\omega$ such that $S\cap S'=\emptyset$. As in the proof of Theorem \ref{s3cl4}, there exist $s\in S'$ and $\eta\in\fix_{\mathcal{G}}(A\setminus A_{s})$ such that $\eta(x)\ne x$. We note that $\fix_{\mathcal{G}}(A\setminus A_{s})$ is isomorphic to $\Aut(\mathbf{A}_{s})$, and thus isomorphic to $\Aut(\mathbf{2}^{\omega})$.

Let
$$Y = \{ \sigma(x) : \sigma \in \fix_{\mathcal{G}}(A\setminus A_{s}) \}.$$
Since $x\in X$, $\fix_{\mathcal{G}}(A\setminus A_{s})\subseteq\fix_{\mathcal{G}}(E)$ and $E$ is a support of $X$, we infer that $Y\subseteq X$. Note that $|Y|\ge 2$ since $x,\eta(x)\in Y$ and $\eta(x)\ne x$.

Furthermore, $Y$ is well-orderable in $\mathcal{N}_{\mathbf{2}^{\omega}}$. Indeed, $E\cup E'$ is a support of every element of $Y$. To see this, let $\rho\in\fix_{\mathcal{G}}(E\cup E')$ and $\sigma\in \fix_{\mathcal{G}}(A\setminus A_{s})$. By the definition of $Y$, it suffices to show that $\rho\sigma$ and $\sigma$ agree on the support $E\cup E'$ of $x$. Let $a\in E\cup E'$. If $a\notin A_{s}$, then clearly $\rho\sigma(a)=a=\sigma(a)$. If $a\in A_{s}$, then $\sigma(a)\in A_{s}\subseteq E\cup E'$ (because, if $n\in\omega$, then every element of $\mathcal{G}$ fixes $A_{n}$), and since $\rho\in\fix_{\mathcal{G}}(E\cup E')$, we have $\rho\sigma(a)=\rho(\sigma(a))=\sigma(a)$. Therefore, $\rho\sigma$ and $\sigma$ agree on  $E\cup E'$, and hence $\rho(\sigma(x))=\rho\sigma(x)=\sigma(x)$. Thus, $Y$ is well-orderable by Fact \ref{f:2}.

To complete the proof, it suffices to show that $Y$ is infinite. By way of contradiction, we assume $Y$ is finite. Then the group $\Sym(Y)$ is also finite.

We define a map $\phi:\fix_{\mathcal{G}}(A\setminus A_{s})\rightarrow\Sym(Y)$ by: 
$$\phi(\pi)(y)=\pi(y)\text{ for all $\pi\in\fix_{\mathcal{G}}(A\setminus A_{s})$ and $y\in Y$.}$$ 
Clearly, $\phi$ is a homomorphism, and hence, by the First Isomorphism Theorem of algebra, $\ker(\phi)$ is a normal subgroup of $\fix_{\mathcal{G}}(A\setminus A_{s})$, and the quotient group $\fix_{\mathcal{G}}(A\setminus A_{s})/\ker(\phi)$ embeds into $\Sym(Y)$. As $\eta\in \fix_{\mathcal{G}}(A\setminus A_{s})\setminus\ker(\phi)$, we have that $\ker(\phi)$ is a proper subgroup of $\fix_{\mathcal{G}}(A\setminus A_{s})$. Since $\fix_{\mathcal{G}}(A\setminus A_{s})$ is isomorphic to $\Aut(\mathbf{2}^{\omega})$, it is a simple group by Theorem \ref{thm:prel2}. We conclude that the group $\ker(\phi)$ is trivial.  Therefore, $\fix_{\mathcal{G}}(A\setminus A_{s})/\ker(\phi)$ is isomorphic to $\fix_{\mathcal{G}}(A\setminus A_{s})$. This implies that $\Sym(Y)$ contains a copy of $\fix_{\mathcal{G}}(A\setminus A_{s})$. But this is impossible because $\Sym(Y)$ is finite and $\fix_{\mathcal{G}}(A\setminus A_{s})$ is infinite. The contradiction obtained shows that $Y$ is infinite, as required.
\vskip.1in

\textbf{B.} Let $X\in\mathcal{N}_{\mathbf{2}^{\omega}}$, $E=\bigcup\{A_{n}: n\in S\}$ for some $S\in [\omega]^{<\omega}$, $x\in X$, $s\in\omega\setminus S$, $\eta\in\fix_{\mathcal{G}}(A\setminus A_{s})$ with $\eta(x)\ne x$, and 
$$Y = \{ \sigma(x) : \sigma \in \fix_{\mathcal{G}}(A\setminus A_{s})\}$$ 
be as in proof \textbf{A}. In view of the argument in the third paragraph of the proof of Theorem \ref{s3cl4},  without loss of generality, we may consider $E\cup A_{s}$ as a support of $x$,\footnote{If $E'=\bigcup\{A_{n}:n\in S'\}$, for some finite $S'\subsetneq\omega$ with $S\subseteq S'$ and $|S'\setminus S|\geq 2$, is a support of $x$, and $s\in S'\setminus S$ and $\eta$ are as above, then we may henceforth replace every occurrence of $E$ by $E'\setminus A_{s}$ and the argument goes through without any other alterations; note that $E'\setminus A_s$ is a support of $X$ since it contains the support $E$ of $X$.} and thus, as a support of $\phi(x)$ for all $\phi\in \mathcal{G}$ since, if $E\cup A_{s}$ is a support of $x$, then, for every $\phi\in \mathcal{G}$, 
$$\phi(E\cup A_{s})=\phi((\bigcup\{A_{n}:n\in S\})\cup A_{s})=(\bigcup\{A_{n}:n\in S\})\cup A_{s}=E\cup A_{s}$$ 
is a support of $\phi(x)$.   

As in proof \textbf{A}, $Y$ is a subset of $X$ such that $Y$ is well-orderable in $\mathcal{N}_{\mathbf{2}^{\omega}}$  (for $E\cup A_{s}$ is a support of every element of $Y$). For the sake of simplicity in notation, let us denote the subgroup $\fix_{\mathcal{G}}(A\setminus A_{s})$ of $\mathcal{G}$ by $\mathbf{G}$. So, under this notation, $Y=\{\sigma(x) : \sigma \in \mathbf{G}\}$. We let
$$H=\{\sigma\in\mathbf{G}:\sigma(x)=x\}.$$
Then $H$ is a proper subgroup of $\mathbf{G}$ because $\eta\in\mathbf{G}\setminus H$. We now prove by contradiction that $Y$ is infinite. So, assume $Y$ is finite. Then the index of $H$ in $\mathbf{G}$ is finite (since $|\mathbf{G}/H|=|Y|$), and thus less than $2^{\aleph_{0}}$.

As $\mathbf{G}$ is isomorphic to $\Aut(\mathbf{2}^{\omega})$, it follows from Theorem \ref{thm:prel3} that there is a finite partition $F$ of $A_{s}$ into clopen sets of $\mathbf{A}_s$, such that: 
$$\fix_{\mathbf{G}}(F){\leq H}{\leq\fix_{\mathbf{G}}(\{F\}}).$$ 
By the first of the above two inequalities and the definition of $H$, we obtain the following: 
\begin{equation}
\label{eq:F_supports_x}
(\forall\phi\in\fix_{\mathbf{G}}(F))(\phi(x)=x).
\end{equation} 
Note that $F$ has at least two members; otherwise, $F=\{A_{s}\}$ which, together with (\ref{eq:F_supports_x}), yields $H=\mathbf{G}$, contradicting the fact that $H$ is a proper subgroup of $\mathbf{G}$. Suppose 
$$F=\{C_1,C_2,\ldots,C_{n}\}$$
for an integer $n\geq 2$ such that the map $\{1,2, \dots, n\}\ni i\mapsto C_{i}$ is a bijection onto $F$. We define a binary relation $f$ by:
$$f=\{\langle\phi(C_1,C_2,\ldots,C_n),\phi(x)\rangle:\phi\in\mathbf{G}\}.$$
Then $f$ has the following properties:
\begin{enumerate}
\item $f\in \mathcal{N}_{\mathbf{2}^{\omega}}$. To show this, we observe that $$f=\{\langle\phi(C_1,C_2,\ldots,C_n),\phi(x)\rangle:\phi\in\fix_{\mathcal{G}}(E)\}.$$ 
Indeed, since $\mathbf{G}\subseteq\fix_{\mathcal{G}}(E)$, $f\subseteq\{\langle\phi(C_1,C_2,\ldots,C_n),\phi(x)\rangle:\phi\in\fix_{\mathcal{G}}(E)\}$. Conversely,
for a given $\phi\in\fix_{\mathcal{G}}(E)$, let $\sigma\in\mathbf{G}$ be such that $\sigma$ agrees with $\phi$ on $A_{s}$, and thus, also on the support $E\cup A_{s}$ of $x$, as well as on the family $\mathcal{R}_{s}$ of all clopen subsets of $A_s$. It follows that $\langle\phi(C_1,C_2,\ldots,C_n),\phi(x)\rangle=\langle\sigma(C_1,C_2,\ldots,C_n),\sigma(x)\rangle\in f$. Hence, the above equality holds; thus, $E$ is a support of $f$, so $f\in\mathcal{N}_{\mathbf{2}^{\omega}}$.\footnote{Note that $f$ is well-orderable in $\mathcal{N}_{\mathbf{2}^{\omega}}$ since $\fix_{\mathcal{G}}(E\cup A_{s})\subseteq\fix_{\mathcal{G}}(f)$.}  
 
\item $f$ is a function. Indeed, let $\phi,\psi\in\mathbf{G}$ be such that $\phi(C_1,C_2,\ldots,C_n)=\psi(C_1,C_2,\ldots,C_n)$. Then, for every $i\in\{1,\dots, n\}$,  $\phi(C_i)=\psi(C_i)$, so $\psi^{-1}\phi\in\fix_{\mathbf{G}}(F)$. By (\ref{eq:F_supports_x}), we obtain that $\psi^{-1}\phi(x)=x$, or equivalently $\phi(x)=\psi(x)$. Thus, $f$ is a function.

\item $\dom(f)$ is denumerable. Indeed, since $\mathbf{G}$ is isomorphic to $\Aut(\mathbf{A}_{s})$, it follows (by the definition of $f$) that $\dom(f)$ is an infinite subset of the set $\mathcal{D}$ of all finite sequences of elements of $\mathcal{R}_{s}$, each having length $n$. As $|\mathcal{R}_{s}|=\aleph_{0}$, it follows that $|\mathcal{D}|=\aleph_{0}$, and hence $|\dom(f)|=\aleph_{0}$.

\item $\ran(f)=\{\phi(x):\phi\in\mathbf{G}\}=Y$.

\item $f$ is finite-to-one. Indeed, let $w\in\ran(f)$. Then $w=\phi(x)$ for some $\phi\in\mathbf{G}$. Let $$Z_{\phi}=\Sym(\phi(F))=\Sym(\{\phi(C_1),\phi(C_{2}),\ldots,\phi(C_{n})\}).$$ 
Clearly, $Z_{\phi}$ is finite; in particular, as $|\phi(F)|=|F|=n$, $|Z_{\phi}|=n!$. We denote the elements of $Z_{\phi}$ as $n$-tuples $\langle\phi(C_{\pi(1)}),\phi(C_{\pi(2)}),\ldots,\phi(C_{\pi(n)})\rangle$ with $\pi\in\Sym(\{1,2,\ldots,n\})$.

We prove that $f^{-1}(\{\phi(x)\})\subseteq Z_{\phi}$, i.e., $f^{-1}(\{w\})\subseteq Z_{\phi}$. Let $\psi\in\mathbf{G}$ be such that $\psi(C_1,C_2,\ldots,C_n)\in f^{-1}(\{\phi(x)\})$. Then we have $\psi(x)=f(\psi(C_1,C_2,\ldots,C_n))=\phi(x)$; thus 
$\psi(x)=\phi(x)$. The latter equality yields $\psi^{-1}\phi\in H$. Hence, since $H\leq\fix_{\mathbf{G}}(\{F\})$, $\psi^{-1}\phi(F)=F$, i.e., $\psi^{-1}\phi$ fixes $F$ setwise. This, together with the fact that $\psi^{-1}\phi$ is one-to-one, means that 
$$\psi^{-1}\phi(C_1,C_2,\ldots,C_n)=\langle C_{\sigma(1)},C_{\sigma(2)},\ldots,C_{\sigma(n)}\rangle$$
for some $\sigma\in\Sym(\{1,2,\ldots,n\})$. It follows that
$$\langle \phi(C_1),\phi(C_2),\ldots,\phi(C_n)\rangle=\langle \psi(C_{\sigma(1)}),\psi(C_{\sigma(2)}),\ldots,\psi(C_{\sigma(n)})\rangle,$$
and, in consequence, by the latter equation, we have:
\begin{align*}
\psi(C_1,C_2,\ldots,C_n)&=\langle \psi(C_{1}),\psi(C_{2}),\ldots,\psi(C_{n})\rangle\\
&=\langle \phi(C_{\sigma^{-1}(1)}),\phi(C_{\sigma^{-1}(2)}),\ldots,\phi(C_{\sigma^{-1}(n)})\rangle\in Z_{\phi}.
\end{align*}
Hence $f^{-1}(\{\phi(x)\})\subseteq Z_{\phi}$, so  $f^{-1}(\{w\})\subseteq Z_{\phi}$. Since $w$ is an arbitrary element of $\ran(f)$, we conclude that $f$ is finite-to-one, as required.
\end{enumerate}

By item 4 and our assumption on $Y$, we deduce that $\ran(f)$ is finite, and since (by item 5) $f$ is finite-to-one, $\dom(f)$ is also finite. But this contradicts item 3. The contradiction obtained proves that $Y$ is infinite.
\end{proof}

\begin{claim}
\label{c:M_features}
If $x\in\mathcal{N}_{\mathbf{2}^{\omega}}$, then the $\mathcal{G}$-orbit of $x$, that is, the set $\Orb_{\mathcal{G}}(x)=\{\phi(x):\phi\in \mathcal{G}\}$, is of size at most $2^{\aleph_{0}}$. In particular, every set in $\mathcal{N}_{\mathbf{2}^{\omega}}$ has a well-orderable partition into well-orderable sets, each of size at most $2^{\aleph_{0}}$.
\end{claim}
\begin{proof}
Fix $x\in\mathcal{N}_{\mathbf{2}^{\omega}}$. Let $S\in [\omega]^{<\omega}\setminus\{\emptyset\}$ be such that  $E=\bigcup\{A_{n}:n\in S\}$ is a support of $x$. We assert that 
\begin{equation}
\label{eq:Orb}
\Orb_{\mathcal{G}}(x)=\Orb_{\fix_{\mathcal{G}}(A\setminus E)}(x)=\{\phi(x):\phi\in\fix_{\mathcal{G}}(A\setminus E)\}.
\end{equation}
Clearly, $\Orb_{\fix_{\mathcal{G}}(A\setminus E)}(x)\subseteq\Orb_{\mathcal{G}}(x)$. Conversely, let $\phi\in \mathcal{G}$. Let $\eta$ be the permutation of $A$ which agrees with $\phi$ on $E$ and is the identity on $A\setminus E$. Since $\phi$ and $\eta$ agree on $E$, it follows that $\eta^{-1}\phi\in\fix_{\mathcal{G}}(E)$, and since $E$ is a support of $x$, $\eta^{-1}\phi(x)=x$ or, equivalently, $\phi(x)=\eta(x)$. As $\eta(x)\in\Orb_{\fix_{\mathcal{G}}(A\setminus E)}(x)$, it follows that $\phi(x)\in\Orb_{\fix_{\mathcal{G}}(A\setminus E)}(x)$. Hence $\Orb_{\mathcal{G}}(x)\subseteq \Orb_{\fix_{\mathcal{G}}(A\setminus E)}(x)$ and, therefore, $\Orb_{\mathcal{G}}(x)=\Orb_{\fix_{\mathcal{G}}(A\setminus E)}(x)$, as asserted.

Let us notice that the group $\fix_{\mathcal{G}}(A\setminus E)$ is isomorphic to the group $\prod\limits_{n\in S}\Aut(\mathbf{A}_{n})$. For every $n\in\omega$, it holds in $\mathcal{N}_{\mathbf{2}^{\omega}}$ that $\mathbf{A}_n$ is homeomorphic to the Cantor cube $\mathbf{2}^{\omega}$, which implies that $|\Aut(\mathbf{A}_{n})|=2^{\aleph_{0}}$ in $\mathcal{N}_{\mathbf{2}^{\omega}}$. Therefore, since $S$ is a non-empty finite set, we have the following equalities in $\mathcal{N}_{\mathbf{2}^{\omega}}$:  $$|\fix_{\mathcal{G}}(A\setminus E)|=|\prod\limits_{n\in S}\Aut(\mathbf{A}_{n})|=2^{\aleph_{0}}.$$
Since the map $\fix_{\mathcal{G}}(A\setminus E)\ni\phi\mapsto\phi(x)$ is a surjection of $\fix_{\mathcal{G}}(A\setminus E)$ onto $\Orb_{\fix_{\mathcal{G}}(A\setminus E)}(x)$, we conclude that, in $\mathcal{N}_{\mathbf{2}^{\omega}}$, $|\Orb_{\fix_{\mathcal{G}}(A\setminus E)}(x)|\leq 2^{\aleph_{0}}$, and thus, by (\ref{eq:Orb}), $|\Orb_{\mathcal{G}}(x)|\leq 2^{\aleph_{0}}$.

For the second assertion of the claim, let us notice that
$$x=\bigcup\{\Orb_{\fix_{\mathcal{G}}(E)}(y):y\in x\}.$$
The family $\mathcal{O}=\{\Orb_{\fix_{\mathcal{G}}(E)}(y):y\in x\}$ is a partition of $x$ and it is well-orderable in $\mathcal{N}_{\mathbf{2}^{\omega}}$ because $E$ is a support of every member of $\mathcal{O}$. Furthermore, by the first part of the proof, every member of $\mathcal{O}$ is of size at most $2^{\aleph_{0}}$. Since $\mathbb{R}$ is well-orderable in every Fraenkel-Mostowski model (for $\mathbb{R}$ is a pure set), it follows that every member of $\mathcal{O}$ is well-orderable in $\mathcal{N}_{\mathbf{2}^{\omega}}$.

The above arguments complete the proof of the claim.
\end{proof}

\begin{claim}
\label{c:WOACfin_in_M}
$\mathcal{N}_{\mathbf{2}^{\omega}}\models\mathbf{WOAC}_{fin}$.
\end{claim}
\begin{proof}
The proof is fairly similar to the one of Theorem \ref{s3cl4}, using this time the part of the argument of either \textbf{A} and \textbf{B} of the proof of Claim \ref{c:IDI_in_M}, which establishes that the suitable, for the current proof, set $Y=\{ \sigma(x) : \sigma \in \fix_{\mathcal{G}}(A\setminus A_{s})\}$ is infinite. We thus take the liberty to leave the details to the interested readers.  
\end{proof}

\begin{claim}
\label{c:CUC_false_in_M}
$\mathcal{N}_{\mathbf{2}^{\omega}}\models\neg\mathbf{CMC}_{\omega}$.
\end{claim}
\begin{proof}
For each $n\in\omega$, we let $\mathcal{C}_{n}$ be the set of all clopen sets in $\mathbf{A}_n=\langle A_{n},\tau_{n}\rangle$, which are in $\mathcal{N}_{\mathbf{2}^{\omega}}$ and are neither $\emptyset$ nor $A_n$. Since $\tau_{n}=\tau_{n}^{\mathcal{M}}$, it follows that $\mathcal{C}_{n}=\{K\in \mathcal{M}: K \text{ is clopen in }\langle A_{n},\tau_{n}\rangle\text{ and }\emptyset\neq K\neq A_n\}=\mathcal{R}_n\setminus\{\emptyset, A_n\}$ (see proof $\mathbf{B}$ of Claim \ref{c:IDI_in_M}). Since, for every $n\in\omega$,  $\mathcal{C}_{n}$ is denumerable in $\mathcal{M}$ and $A_{n}$ is a support of every element of $\mathcal{C}_{n}$, we conclude that $\mathcal{C}_{n}$ is denumerable in $\mathcal{N}_{\mathbf{2}^{\omega}}$. 

We let
$$\mathcal{C}=\{\mathcal{C}_{n}:n\in\omega\}.$$
Then $\mathcal{C}$ is denumerable in $\mathcal{N}_{\mathbf{2}^{\omega}}$ since, for every $\phi\in G$ and every $n\in\omega$, $\phi(\mathcal{C}_{n})=\mathcal{C}_{n}$. (Recall that, for every $\phi\in G$ and every $n\in\omega$, $\phi\upharpoonright A_{n}\in\Aut(\mathbf{A}_{n})$.)

We prove that $\mathcal{C}$ has no partial multiple choice function in $\mathcal{N}_{\mathbf{2}^{\omega}}$, which will yield that $\mathbf{CMC}_{\omega}$ is false in $\mathcal{N}_{\mathbf{2}^{\omega}}$. By way of contradiction, we assume that, in $\mathcal{N}_{\mathbf{2}^{\omega}}$,  $\mathcal{C}$ has an infinite subfamily $\mathcal{B}$ such that some $f\in\mathcal{N}_{\mathbf{2}^{\omega}}$ is a multiple choice function of $\mathcal{B}$. Let $E=\bigcup\{A_{n}:n\in S\}$, for some finite $S\subsetneq\omega$, be a support of $f$. Since $\mathcal{B}$ is infinite and $S$ is finite, there exists $m\in\omega$ such that $\mathcal{C}_{m}\in\mathcal{B}$ and $A_{m}\cap E=\emptyset$. As $\mathcal{C}_{m}\in\mathcal{B}$ and $f$ is a multiple choice function for $\mathcal{B}$, $f(\mathcal{C}_{m})$ is a non-empty, finite subset of $\mathcal{C}_{m}$. Hence, since $\mathcal{C}_m$ is infinite, we can fix $W\in\mathcal{C}_m\setminus f(\mathcal{C}_m)$. Since $\emptyset\neq f(\mathcal{C}_m)$, we can also fix $C\in f(\mathcal{C}_m)$. 

By Theorem \ref{thm:prel1}, there exists $\phi\in\Aut(\mathbf{A}_{m})$ such that $\phi(C)=W$. Let $\psi$ be the permutation of $A$ which agrees with $\phi$ on $A_{m}$ and is the identity on $A\setminus A_{m}$. Since $E\cap A_m=\emptyset$, it follows that  $\psi\in\fix_{\mathcal{G}}(E)$. Then $\psi(f)=f$ because $E$ is a support of $f$. Furthermore, since $C\in f(\mathcal{C}_{m})$, we have  $W=\phi(C)=\psi(C)\in\psi(f(\mathcal{C}_{m}))$. Therefore, $\psi(f(\mathcal{C}_{m}))\neq f(\mathcal{C}_{m})$ because $W\in\psi(f(\mathcal{C}_{m}))\setminus f(\mathcal{C}_{m})$. On the other hand, since $f$ is a function and the following implications are true:
\begin{multline*}
\langle\mathcal{C}_{m},f(\mathcal{C}_{m})\rangle\in f\rightarrow\psi(\mathcal{C}_{m},f(\mathcal{C}_{m}))\in\psi(f)\rightarrow\langle\psi(\mathcal{C}_{m}),\psi(f(\mathcal{C}_{m}))\rangle\in f\\
\rightarrow\langle\mathcal{C}_{m},\psi(f(\mathcal{C}_{m}))\rangle\in f,
\end{multline*}
we deduce that $f(\mathcal{C}_{m})=\psi(f(\mathcal{C}_{m}))$, which is a contradiction. Thus, $\mathcal{C}$ has no partial multiple choice function in $\mathcal{N}_{\mathbf{2}^{\omega}}$, as required. 
\end{proof}

\begin{claim}
\label{c:MCWO_in_M}
$\mathcal{N}_{\mathbf{2}^{\omega}}\models \mathbf{M}(C(\not\leq\aleph_{0}),\geq 2^{\aleph_{0}})$.
\end{claim}
\begin{proof}
Let $\mathbf{X}=\langle X,d\rangle$ be an uncountable compact metric space in $\mathcal{N}_{\mathbf{2}^{\omega}}$. If $X$ is well-orderable in $\mathcal{N}_{\mathbf{2}^{\omega}}$, and thus (by \cite[Theorem 2.1]{kt}) separable in $\mathcal{N}_{\mathbf{2}^{\omega}}$, then $\mathcal{N}_{\mathbf{2}^{\omega}}\models|X|=2^{\aleph_{0}}$ and we are done. So, assume $X$ is not well-orderable in $\mathcal{N}_{\mathbf{2}^{\omega}}$.  
Let, for some finite $S\subsetneq\omega$, the set $E=\bigcup\{A_{s}:s\in S\}$ be a support of $\mathbf{X}$.   

As in proof \textbf{B} of Claim \ref{c:IDI_in_M}, let $x$ be an element of $X$ which does not have $E$ as a support, $s\in\omega\setminus S$ be such that $E\cup A_{s}$ is (without loss of generality) a support of $x$, $\eta\in\mathbf{G}$ (where $\mathbf{G}=\fix_{\mathcal{G}}(A\setminus A_{s})$) be such that $\eta(x)\ne x$, and let 
$$Y=\{\sigma(x):\sigma\in\mathbf{G}\}=\Orb_{\fix_{\mathcal{G}}(E)}(x)$$ 
(where the second of the above equalities follows from the fact that $E\cup A_{s}$ is a support of $x$--see the argument in item 1 of proof \textbf{B}) be the infinite (by proof \textbf{B}) subset of $X$ which is well-orderable in $\mathcal{N}_{\mathbf{2}^{\omega}}$ (recall $\fix_{\mathcal{G}}(E\cup A_{s})\subseteq\fix_{\mathcal{G}}(Y)$). 
By Claim \ref{c:M_features}, we have $|Y|=|\Orb_{\fix_{\mathcal{G}}(E)}(x)|\leq 2^{\aleph_{0}}$. Furthermore, as $\mathbf{X}$ is compact, 
$$\ran(d\upharpoonright Y\times Y)\text{ is infinite.}$$
Otherwise, it is fairly easy to verify that the infinite set $Y$ would be a discrete, closed subset of $\mathbf{X}$, contradicting the compactness of $\mathbf{X}$.
 
We assert that $|Y|=2^{\aleph_{0}}$. If not, then $|Y|<2^{\aleph_{0}}$; thus, since for the (proper) subgroup $$H=\{\sigma\in\mathbf{G}:\sigma(x)=x\}$$ of $\mathbf{G}$, we have $|\mathbf{G}/H|=|Y|$, the index of $H$ in $\mathbf{G}$ is strictly less than $2^{\aleph_{0}}$.

Let, by Theorem \ref{thm:prel3}, $F=\{C_1,C_2,\ldots,C_n\}$ be a finite partition of $A_{s}$ into clopen sets of $\mathbf{A}_s$, such that $\fix_{\mathbf{G}}(F)\leq H\leq\fix_{\mathbf{G}}(\{F\})$ (we recall that, by proof \textbf{B}, $|F|\geq 2$), and also let $$f=\{\langle\phi(C_1,C_2,\ldots,C_n),\phi(x)\rangle:\phi\in\fix_{\mathcal{G}}(E)\}$$ be the function defined in proof \textbf{B} of Claim \ref{c:IDI_in_M}, where it has been shown that $E$ is a support of $f$, and $f$ is a finite-to-one function from the denumerable set $\Orb_{\fix_{\mathcal{G}}(E)}(\langle C_{1},C_{2},\ldots,C_{n}\rangle)$ onto $Y$; moreover, for all $\phi\in\fix_{\mathcal{G}}(E)$, we have $f^{-1}(\{\phi(x)\})\subseteq \Sym(\phi(F))$ (see items 1 and 5 of proof \textbf{B}). So,
\begin{multline}
\label{eq:dom(f)_1}
\dom(f)=\Orb_{\fix_{\mathcal{G}}(E)}(\langle C_{1},C_{2},\ldots,C_{n}\rangle)=\bigcup\{f^{-1}(\{\phi(x)\}):\phi\in\fix_{\mathcal{G}}(E)\}\\=\bigcup\Orb_{\fix_{\mathcal{G}}(E)}(f^{-1}(\{x\}))
\end{multline}
and $\Orb_{\fix_{\mathcal{G}}(E)}(f^{-1}(\{x\}))=\{f^{-1}(\{\phi(x)\}):\phi\in\fix_{\mathcal{G}}(E)\}$ is a denumerable, disjoint family of sets, each having cardinality at most $n!$.\footnote{We also note that $\dom(f)=\bigcup\Orb_{\fix_{\mathcal{G}}(E)}(\Sym(F))$. Indeed, by the first equality in (\ref{eq:dom(f)_1}), it is clear that $\dom(f)\subseteq\bigcup\Orb_{\fix_{\mathcal{G}}(E)}(\Sym(F))$. For the reverse inclusion, fix $y\in\bigcup\Orb_{\fix_{\mathcal{G}}(E)}(\Sym(F))$. There exist $\phi\in\fix_{\mathcal{G}}(E)$, $\sigma\in\Sym(\{1,\ldots,n\})$ such that $y=\phi(C_{\sigma(1)},\ldots,C_{\sigma(n)})$. Since $F$ comprises clopen sets in the Cantor space $\mathbf{A}_{s}$, any two distinct elements of $F$ are, by Theorem \ref{thm:prel1}, homeomorphic, and so, by \cite[Theorem 7.3 (The pasting lemma)]{m}, there exists $\psi\in\mathbf{G}\subseteq\fix_{\mathcal{G}}(E)$ such that $\psi(C_1,\ldots,C_n)=\langle C_{\sigma(1)},\ldots,C_{\sigma(n)}\rangle$. Hence, $y=\phi(C_{\sigma(1)},\ldots,C_{\sigma(n)})=\phi\psi(C_1,\ldots,C_n)\in\dom(f)$, so $\bigcup\Orb_{\fix_{\mathcal{G}}(E)}(\Sym(F))\subseteq\dom(f)$. Thus, as $\langle C_1,\ldots,C_n\rangle$ is an $n$-element, ordered partition of ${A}_{s}$ into clopen sets of $\mathbf{A}_s$, the above equation yields that $\dom(f)$ is the set of \emph{all} $n$-element, ordered partitions of ${A}_{s}$ into clopen sets of $\mathbf{A}_s$.}
  
Based on the definition of $f$, equation (\ref{eq:dom(f)_1}) and the above fact about $\Orb_{\fix_{\mathcal{G}}(E)}(f^{-1}(\{x\}))$, as well as on the fact that $d$ is a metric, we define a pseudometric $\rho$ on $\dom(f)$ by:
$$\rho(u,v)=
\begin{cases}
0, &\text{if $f(u)=f(v)$;}\\
d(f(u),f(v)), &\text{if $f(u)\ne f(v)$.}
\end{cases}$$
If $u,v\in\dom(f)$ belong to the same element of the orbit $\Orb_{\fix_{\mathcal{G}}(E)}(f^{-1}(\{x\}))$, then their $\rho$-distance is zero, and if they belong to distinct elements of $\Orb_{\fix_{\mathcal{G}}(E)}(f^{-1}(\{x\}))$, then their (positive) $\rho$-distance is $d(f(u),f(v))$. By the definition of $\rho$, the second (or the third) equality of (\ref{eq:dom(f)_1}), and the fact that $E$ is a support of both $f$ and $d$, it follows that $E$ is also a support of $\rho$, and thus $\rho\in \mathcal{N}_{\mathbf{2}^{\omega}}$. Furthermore, as $\ran(\rho)=\ran(d\upharpoonright Y\times Y)$ and the latter set is infinite, $\ran(\rho)$ is infinite. Note that, essentially, the metric space $\langle Y,d\upharpoonright Y\times Y\rangle=\langle \ran(f),d\upharpoonright \ran(f)\times \ran(f)\rangle$ is, in $\mathcal{N}_{\mathbf{2}^{\omega}}$, homeomorphic to the metric identification of $\langle\dom(f),\rho\rangle$.

Let
$$\mathcal{Z}=\{\Orb_{\fix_{\mathcal{G}}(E)}(\langle y,z\rangle):y,z\in \dom(f)\text{ and }f(y)\neq f(z)\}.$$
Clearly, 
\begin{equation}
\label{eq:Z}
\bigcup\mathcal{Z}=(\dom(f)\times\dom(f))\setminus\{\langle y,z\rangle:f(y)=f(z)\}
\end{equation}
and note that
\begin{equation}
\label{eq:equal_dist}
(\forall Z\in\mathcal{Z})(\forall \langle y_1,z_1\rangle,\langle y_2,z_2\rangle\in Z)( \rho(y_1,z_1)=\rho(y_2,z_2)).
\end{equation}
Indeed, let $Z\in\mathcal{Z}$ and $\langle y_1,z_1\rangle,\langle y_2,z_2\rangle\in Z$. There exists $\phi\in\fix_{\mathcal{G}}(E)$ such that $\langle y_2,z_2\rangle=\phi(\langle y_1,z_1\rangle)$. Since $\phi\in\fix_{\mathcal{G}}(E)$ and $E$ is a support of $\rho$, it follows that $\rho(y_1,z_1)=\rho(y_2,z_2)$.

We show that $\mathcal{Z}$ is finite. This, together with (\ref{eq:Z}), (\ref{eq:equal_dist}) and the definition of $\rho$, will give us that $\ran(\rho)$ is finite, which is a contradiction. Let us first denote (for the sake of simplicity) the $n$-tuple $\langle C_1, C_2,\ldots,C_n\rangle$ by $\mathbf{C}$. 

Let $U$ be a partition of ${A}_s$ into clopen sets of $\mathbf{A}_s$ such that $|U|=n$ and, for every $T\in U$ and every  $i\in\{1,\ldots,n\}$, $T\cap C_{i}\neq\emptyset$. Then  $U\cap F=\emptyset$, and, for every $T\in U$, $T=\bigcup\limits_{i=1}^{n}(T\cap C_{i})$ because $F=\{C_1,C_2,\ldots,C_n\}$ is a partition of $A_{s}$. Let $U^{*}=\{T\cap C_{i}:T\in U$ and $1\leq i\leq n\}$ and $U^{**}=\{\bigcup\mathcal{Q}:\mathcal{Q}\subseteq U^{*}\}\setminus\{\emptyset\}$. Note that $F\cup U\cup U^{*}\subseteq U^{**}$, so $|U^{**}|>2n=2|F|$.

Let 
$$V=\{\langle\phi(\mathbf{C}),\psi(\mathbf{C})\rangle:\phi,\psi\in\fix_{\mathcal{G}}(E)\text{ and }\phi(F)\cup\psi(F)\subseteq U^{**}\}.$$  
Since $(U^{**})^{F}$ (the set of all functions from $F$ into $U^{**}$) is finite, it is clear that $V$ is finite. Hence, the set
$$O=\{\Orb_{\fix_{\mathcal{G}}(E)}(\langle u,v\rangle):\langle u,v\rangle\in V\}$$
is also finite. 

We assert that $\mathcal{Z}\subseteq O$. Let $Z\in\mathcal{Z}$. Then, $Z=\Orb_{\fix_{\mathcal{G}}(E)}(\langle y,z\rangle)$ for some $y,z\in\dom(f)$ with $f(y)\neq f(z)$. By the definition of $f$, there exist $\phi,\psi\in\fix_{\mathcal{G}}(E)$ such that $y=\phi(\mathbf{C})$ and $z=\psi(\mathbf{C})$. As $|U^{**}|>2|F|$, we have $|\phi(F)\cup\psi(F)|<|U^{**}|$, so using Theorem \ref{thm:prel1} and \cite[Theorem 7.3 (The pasting lemma)]{m}, we may construct a $\sigma\in\mathbf{G}\subseteq\fix_{\mathcal{G}}(E)$ such that $\sigma(\phi(F))\cup\sigma(\psi(F))\subseteq U^{**}$. Thus, $\langle\sigma(y),\sigma(z)\rangle=\langle\sigma\phi(\mathbf{C}),\sigma\psi(\mathbf{C})\rangle\in V$, so $Z=\Orb_{\fix_{\mathcal{G}}(E)}(\langle y,z\rangle)=\Orb_{\fix_{\mathcal{G}}(E)}(\langle \sigma (y),\sigma (z)\rangle)\in O$. Therefore, $\mathcal{Z}\subseteq O$, and as $O$ is finite, so is $\mathcal{Z}$.

By the above arguments, we conclude that $\ran(\rho)$ is finite, which is a contradiction. 

Thus, $|Y|=2^{\aleph_{0}}$ and, as $Y\subseteq X$, $2^{\aleph_{0}}\leq |X|$ in $\mathcal{N}_{\mathbf{2}^{\omega}}$. It follows that $\mathbf{M}(C(\not\leq\aleph_{0}),\geq 2^{\aleph_{0}})$ is true in $\mathcal{N}_{\mathbf{2}^{\omega}}$, as required.
\end{proof}   

The above claims complete the proof of the theorem. 
\end{proof}

\section{The shortlist of open problems}
\label{s5}

\begin{enumerate}
\item Is $\mathbf{M}(C,S)$ true in the model $\mathcal{N}_{\mathbf{2}^{\omega}}$? (Our conjecture is that the answer to this question is in the affirmative.)
\item Is there a $\mathbf{ZF}$-model for $\mathbf{\Phi}_2$?
\item Is there a $\mathbf{ZF}$-model for $\mathbf{\Phi}_4\wedge\mathbf{\Phi}_5$?
\item Are $\mathbf{M}(C(\not\leq\aleph_{0}),\geq 2^{\aleph_0})$ and $\mathbf{M}(C, S)$  equivalent in $\mathbf{ZF}$? (See Question \ref{s4q1}(i).)
\item Does $\mathbf{CAC}_{fin}$ imply $\mathbf{M}(C(\not\leq\aleph_{0}),\geq 2^{\aleph_0})$ in $\mathbf{ZF}$? (See Question \ref{s4q1}(ii).)
 
\item  Does $\mathbf{CPM}_{le}(C,C)$ imply $\mathbf{CPM}_{le}(C,M)$ in $\mathbf{ZF}$? (See Question \ref{s2q10} (ii).)
\item Does $\mathbf{CPM}_{le}(CS,MS)$ imply $\mathbf{M}(C,S)$ in $\mathbf{ZF}$? (See Question \ref{s2q10} (iii).)

\item Are  $\mathbf{M}(C,S)$, $\mathbf{CAC}(C, \mathbf{M}_{le})$, $\mathbf{CPM}_{le}(C,S)$, $\mathbf{CPM}_{le}(C,C)$ all equivalent in $\mathbf{ZF}$? (Cf. Question \ref{s2q14}.)
\end{enumerate}

\section{The diagram}
\label{s6}

To simplify the diagram illuminating the main implications deduced in this article, we denote by $\mathbf{\Psi}_1$ any form from the equivalent forms listed in Theorem \ref{s2t2}(i), and by $\mathbf{\Psi}_2$ any form from the equivalent forms listed in Theorem \ref{s2t2}(viii); that is:
\begin{multline*}
\mathbf{\Psi}_1\in\{ \mathbf{CPM}_{le}(C,2), 
\mathbf{CPM}_{le}(C,MS), \mathbf{CPM}_{le}(C,M2), \mathbf{CPM}_{le}(C, C2)\\
 \mathbf{CSM}_{le}(C,2), \mathbf{CSM}_{le}(C,MS), \mathbf{CSM}_{le}(C,M2) \},
 \end{multline*}
\begin{multline*}
 \mathbf{\Psi}_2\in\{ \mathbf{CPM}_{le}(CS,MS), \mathbf{CPM}_{le}(CS,2), \mathbf{CPM}%
_{le}(CS,MC),\\ \mathbf{CPM}_{le}(CS,M2), \mathbf{CPM}_{le}(C2,MS),
\mathbf{CPM}_{le}(C2,MC), \mathbf{CPM}_{le}(C2,M2),\\ \mathbf{CSM}%
_{le}(CS,MS), \mathbf{CSM}_{le}(CS,M2), \mathbf{CSM}_{le}(C2,MS), \mathbf{CSM}_{le}(C2,M2)\}.
 \end{multline*}
 
 \begin{center}
 {\small\textbf{Diagram.} Deductive strength of certain forms}
 \bigskip
\begin{tikzpicture}[font=\small]
\node (Psi1) at (-4,11) {$\mathbf{\Psi}_1$};
\node (CMCRCPMleCC) at(-2,12) {($\mathbf{CMC}(\leq 2^{\aleph_0})\wedge\mathbf{CPM}_{le}(C,C))$};
\node (Psi2) at (-5,5.3) {$\mathbf{\Psi}_2$};
\node (CSMleCS) at (1,6.3) {$\mathbf{CSM}_{le}(C,S)$};
\node (CPMleCS) at (-2,6.3) {$\mathbf{CPM}_{le}(C,S)$};
\node (CPMleCMC) at (1,10) {$\mathbf{CPM}_{le}(C, MC)$};
\node (CPMleCMCMCS) at (0,11) {($\mathbf{CPM}_{le}(C, MC)\wedge\mathbf{M}(C,S)$)};
\node (MCS) at (5.6,8) {$\mathbf{M}(C, S)$};
\node (Malephs) at (3,7) {$\mathbf{M}(C(\nleq\aleph_0),\geq2^{\aleph_0})$};
\node (CACfin) at (2.6,5.3) {$\mathbf{CAC}_{fin}$};
\node (CPMleCC) at (2.5,9.2) {$\mathbf{CPM}_{le}(C,C)$};
\node (CACCMle) at (5.5,9) {$\mathbf{CAC}(C,\mathbf{M}_{le})$};
\node (CPMleCSCACR) at (-2,8) {($\mathbf{CPM}_{le}(C, S)\wedge\mathbf{CAC}(\mathbb{R})$)};
\node (CPMleCSM) at (-3.5,2.2) {$\mathbf{CPM}_{le}(CS,M)$};
\node (CPMleCSMCUCfin) at (-3.3, 3.2) {($\mathbf{CPM}_{le}(CS, M)\wedge\mathbf{CUC}_{fin})$};
\node (CUC) at (-1.3,5.3) {$\mathbf{CUC}$};
\node (CPMleCSMCUC) at (-3,4.2) {($\mathbf{CPM}_{le}(CS, M)\wedge\mathbf{CUC}$)};
\node (UTalephscuf) at (1.3,2.2) {$\mathbf{UT}(\aleph_0, \aleph_0,cuf)$};
\node(CMComega) at (6,1.8) {$\mathbf{CMC}_{\omega}$};
\node (CAComega) at (0.5,5.3) {$\mathbf{CAC}_{\omega}$};

\draw [<->] (Psi1)--(CPMleCMCMCS);
\draw [->] (CPMleCMCMCS)--(CPMleCMC);
\draw [->] (CPMleCMC)-- (CPMleCC);
\draw [->] (CPMleCMCMCS)--(CPMleCMC);
\draw [<->] (CPMleCS)--(CSMleCS);
\draw [->] (CPMleCMC)--(CSMleCS);
\draw [->] (CMCRCPMleCC)--(Psi1);
\draw [->] (Psi1)--(Psi2);
\draw [->] (CPMleCSCACR)--(Psi1);
\draw [->] (CPMleCSCACR)--(CPMleCS);
\draw [->] (Psi2)--(CUC);
\draw [<->] (CPMleCSMCUCfin)--(CPMleCSMCUC);
\draw [->] (CPMleCSMCUCfin)--(CPMleCSM);
\draw [->] (CPMleCSM)--(UTalephscuf);
\draw [->] (CPMleCSMCUC)--(CUC);
\draw [->] (CUC)--(UTalephscuf);

\draw [->] (CPMleCC)--(CACCMle);
\draw [->] (CACCMle)--(MCS);
\draw [->] (MCS)--(Malephs);
\draw [->] (Malephs)--(CACfin);
\draw [->] (CAComega)--(CACfin);
\draw [->] (CUC)--(CAComega);

\draw [->] (UTalephscuf)--(CMComega);
\draw [->] (CAComega)--(CMComega);

\draw [->, dashed, postaction={decorate,decoration={
          raise=0.8ex,
          text along path,
          text format delimiters= {|}{|}, 
          text={|\tiny| (ZFA)},
          text align = center,
        }
      }] (Malephs) -- node{/} (CPMleCC);
      
 \draw [->, dashed, postaction={decorate,decoration={
          raise=0.8ex,
          text along path,
          text format delimiters= {|}{|}, 
          text={|\tiny| (ZFA)},
          text align = center,
        }
      }] (Malephs) -- node{/}(CACCMle);

 \draw [->, dashed, postaction={decorate,decoration={
          raise=0.8ex,
          text along path,
          text format delimiters= {|}{|}, 
          text={|\tiny| (ZFA)},
          text align = center,
        }
      }] (Malephs) -- node{/}(CMComega);      

\end{tikzpicture}
\vskip.1in

\end{center}

\end{document}